\documentclass[11pt,reqno]{amsart}
\usepackage[T1]{fontenc}
\usepackage[utf8]{inputenc}
\usepackage{lmodern}
\usepackage{amsmath,amssymb,mathtools,mathrsfs}
\usepackage{microtype}
\usepackage[margin=1in]{geometry}
\usepackage{enumitem}
\usepackage{booktabs}
\usepackage{hyperref}
\hypersetup{hidelinks,pdfcreator={pdfLaTeX},pdfdisplaydoctitle=true,pdftitle={Finite-index constant-mean-curvature hypersurfaces in space forms of dimension at most seven},pdfsubject={Finite-index CMC hypersurfaces, Green-function identities, and space-form rigidity},pdfkeywords={constant mean curvature, finite Morse index, space forms, hyperbolic space, Green function, volume entropy}}
\numberwithin{equation}{section}
\newtheorem{theorem}{Theorem}[section]
\newtheorem{proposition}[theorem]{Proposition}
\newtheorem{lemma}[theorem]{Lemma}
\newtheorem{corollary}[theorem]{Corollary}
\theoremstyle{definition}

\theoremstyle{remark}
\newtheorem{remark}[theorem]{Remark}
\newcommand{\R}{\mathbb R}

\newcommand{\dd}{\,\mathrm d}
\newcommand{\dV}{\,\mathrm dV_g}
\newcommand{\dAg}{\,\mathrm dA_g}

\newcommand{\dmu}{\,\mathrm d\mu}
\newcommand{\ip}[2]{\langle #1,#2\rangle}
\DeclareMathOperator{\Ric}{Ric}
\DeclareMathOperator{\tr}{tr}
\DeclareMathOperator{\Vol}{Vol}

\DeclareMathOperator{\ind}{ind}
\DeclareMathOperator{\divg}{div}

\DeclareMathOperator{\op}{op}
\newcommand{\doi}[1]{\href{https://doi.org/#1}{doi:\nolinkurl{#1}}}
\newcommand{\arxiv}[1]{\href{https://arxiv.org/abs/#1}{arXiv:\nolinkurl{#1}}}

\allowdisplaybreaks[1]
\author{Zihao Wang}
\address{School of Mathematical Sciences, Fudan University, Shanghai 200433, China}
\email{wangzh25@m.fudan.edu.cn}

\date{September 17, 2026}
\title[Finite index CMC hypersurfaces in space forms]{Finite index constant mean curvature hypersurfaces\\ in space forms of dimension at most seven}
\begin{document}
\begin{abstract}
We study complete two-sided constant-mean-curvature hypersurfaces of dimensions $2\le n\le6$ and finite Morse index in simply connected space forms. In the round sphere the immersed domain is compact; in Euclidean space every noncompact example is minimal; in hyperbolic space of curvature $-1$ we obtain compactness under explicit dimension-dependent mean-curvature thresholds, from $H^2>1$ for surfaces to $H^2\ge5/3$ in dimension six. The conclusions also hold for the volume-constrained index, without properness, volume-growth, or curvature-bound assumptions. A common Green-function argument uses explicit rational combinations of eight identities and the stable Bernstein theorem of Hong--Li--Wang. Stable hyperbolic tubes show that no mean-curvature threshold independent of dimension can give compactness in all dimensions.
\end{abstract}
\maketitle

\section{Introduction and main results}\label{sec:introduction}
Let
\[
 X_c^{n+1}=\begin{cases}\mathbb H^{n+1}(-1),&c=-1,\\ \mathbb R^{n+1},&c=0,\\ \mathbb S^{n+1}(1),&c=1.\end{cases}
\]
endowed with its standard metric $\bar g_c$. Fix an integer $2\le n\le6$. Unless another dimension is explicitly specified, an immersion
$F:(M^n,g)\to X_c^{n+1}$ is smooth, $M$ is connected and without boundary, and the induced metric $g=F^*\bar g_c$ is complete. Two-sidedness means that a global unit normal $N$ to the immersion has been chosen. Our sign and mean-curvature conventions are
\begin{equation}\label{eq:conventions}
 A(X,Y)=\langle\bar\nabla_{dF(X)}N,dF(Y)\rangle,
 \qquad H=\tfrac1n\tr_g A,
 \qquad B=A-Hg.
\end{equation}
Thus $H$ is the averaged mean curvature, $B$ is the trace-free second fundamental form, and
\begin{equation}\label{eq:normA}
 a:=|B|,\qquad |A|^2=a^2+nH^2.
\end{equation}
We assume that $H$ is constant. The Laplacian is $\Delta=\divg\nabla$. We use $R(X,Y)Z=\nabla_X\nabla_YZ-\nabla_Y\nabla_XZ-\nabla_{[X,Y]}Z$ and $\sec(X,Y)=\langle R(X,Y)Y,X\rangle$ for orthonormal $X,Y$; hence the unit sphere has curvature $+1$.

For $f\in C_c^\infty(M)$, define the Jacobi quadratic form
\begin{equation}\label{eq:Q}
 Q_c(f)=\int_M\bigl(|\nabla f|^2-(|A|^2+nc)f^2\bigr)\dV.
\end{equation}
The \emph{strong Morse index} $\ind(Q_c)$ is the supremum of the dimensions of subspaces of $C_c^\infty(M)$ on which $Q_c$ is negative definite. Strong stability means $Q_c\ge0$ on this entire test space. In contrast, the \emph{volume-constrained index} $\ind_0(Q_c)$ is the index of the restriction to
\begin{equation}\label{eq:zeroMean}
 \mathcal V_0=\left\{f\in C_c^\infty(M):\int_Mf\dV=0\right\}.
\end{equation}
Weak stability means nonnegativity on $\mathcal V_0$. These are the variational conventions of \cite{BC84,BCE88}; the zero-mean condition is imposed on the immersed domain and does not require embeddedness or a globally enclosed finite volume.

\begin{theorem}[Spherical compactness]\label{thm:sphere}
Let $2\le n\le6$ and let $F:(M^n,g)\to\mathbb S^{n+1}(1)$ be a smooth, connected, complete, two-sided constant-mean-curvature immersion without boundary. If $\ind(Q_1)<\infty$, then $M$ is compact.
\end{theorem}
\begin{theorem}[Euclidean minimality]\label{thm:euclidean}
Let $2\le n\le6$ and let $F:(M^n,g)\to\mathbb R^{n+1}$ be a smooth, connected, complete, noncompact, two-sided constant-mean-curvature immersion without boundary. If $\ind(Q_0)<\infty$, then $H=0$.
\end{theorem}
\begin{theorem}[Hyperbolic compactness at large mean curvature]\label{thm:hyperbolic}
Let $2\le n\le6$ and let $F:(M^n,g)\to\mathbb H^{n+1}(-1)$ be a smooth, connected, complete, two-sided constant-mean-curvature immersion without boundary. If $\ind(Q_{-1})<\infty$ and
\begin{equation}\label{eq:threshold}
 \begin{cases}
 H^2>1,&n=2,\\
 H^2\ge\tau_n,&3\le n\le6,
 \end{cases}
 \qquad
 \begin{array}{c|cccc}
 n&3&4&5&6\\\hline
 \tau_n&251/250&207/200&29/25&5/3
 \end{array}
\end{equation}
then $M$ is compact.
\end{theorem}
The spherical statement includes minimal immersions as the case $H=0$; no restriction on $|H|$ is imposed. Compactness concerns the domain $M$, not merely its image in the compact ambient sphere. None of the three theorems assumes properness, embeddedness, a volume-growth bound, or a bound for the second fundamental form. By homothety, the spherical conclusion holds for every round $(n+1)$-sphere of positive radius. In curvature $-a_0^2<0$, the hyperbolic condition is $H^2>a_0^2$ for $n=2$ and $H^2\ge a_0^2\tau_n$ for $3\le n\le6$.

The constants $\tau_n$ are sufficient values furnished by explicit rational coefficients, not asserted optimal values. The averaged convention \eqref{eq:conventions} is important: for $n\ge3$ the condition is $|\tr A|\ge n\sqrt{\tau_n}$ in the trace convention. In dimension two the strict endpoint is sharp because horospheres have $H=1$ and are complete, noncompact, and strongly stable. Corollary~\ref{cor:tube6} gives a different complete noncompact strongly stable example in $\mathbb H^7(-1)$ with $H^2=49/48>1$; thus the surface threshold cannot be used unchanged throughout this dimension range.

\begin{corollary}[Volume-constrained finite index]\label{cor:constrained}
All three theorems remain valid with $\ind_0(Q_c)<\infty$ in place of $\ind(Q_c)<\infty$.
\end{corollary}
\begin{corollary}[Strong stability]\label{cor:strong}
For $2\le n\le6$, every smooth, connected, complete, boundaryless, two-sided strongly stable CMC immersion into $\mathbb R^{n+1}$ has image an affine hyperplane. There is no such strongly stable CMC immersion into $\mathbb S^{n+1}(1)$, or into $\mathbb H^{n+1}(-1)$ under the relevant condition in \eqref{eq:threshold}.
\end{corollary}
\begin{corollary}[Weak stability in the sphere]\label{cor:weak}
For $2\le n\le6$, every smooth, connected, complete, boundaryless, two-sided weakly stable CMC immersion into $\mathbb S^{n+1}(1)$ is a geodesic $n$-sphere, including the totally geodesic equator.
\end{corollary}
The last statement combines Theorem~\ref{thm:sphere} with the compact classification of Barbosa--do Carmo--Eschenburg \cite[Theorem~1.2]{BCE88}. Finite index alone does not imply that a spherical CMC immersion is minimal or totally umbilical.

\subsection{A common analytic statement}
The bounded-curvature part of the three theorems can be expressed in one analytic statement. Stability on an open set always means strong stability for test functions supported in that set.
\begin{theorem}[Exterior stability with bounded curvature]\label{thm:bounded}
Let $2\le n\le6$ and $c\in\{-1,0,1\}$, and let $F:(M^n,g)\to X_c^{n+1}$ satisfy the standing assumptions. Suppose that $M$ is noncompact,
\[
 \sup_M|A|<\infty,\qquad
 Q_c(f)\ge0\quad\text{for every }f\in C_c^\infty(M\setminus K)
\]
for some compact set $K\subset M$. Then
\[
 \begin{cases}
 c+H^2=0,&c\in\{0,1\},\\
 H^2\le1,&c=-1,\ n=2,\\
 H^2<\tau_n,&c=-1,\ 3\le n\le6.
 \end{cases}
\]
\end{theorem}
The proof separates a geometric reduction from an analytic estimate. Finite index gives stability outside a compact set. A curvature blow-up, together with the stable Bernstein theorem of Hong--Li--Wang, gives a global bound for the second fundamental form. These facts are established in Section~\ref{sec:geometry} before the analytic argument begins. Theorem~\ref{thm:bounded} itself uses only the space-form identities of Section~\ref{subsec:spaceform} and the analysis in Sections~\ref{sec:green}--\ref{sec:cutoff}; it does not use stable minimal rigidity.

Under the contradiction assumption, $c+H^2>0$. Exterior stability and infinite intrinsic volume imply a positive global spectral bottom. A unit-pole positive Green function $G$ then tends to zero at infinity and has bounded logarithmic gradient on a sufficiently low sublevel set. Put
\[
 w=\frac{|\nabla G|}{G},\qquad h=\frac H w,\qquad
 \kappa=\frac c{w^2},\qquad
 \mathrm d\mu=w^4G\,\mathrm dV_g,\qquad
 s=\log\frac{t_0}{G}.
\]
The normalized quantities are defined on the regular Green set and interpreted in their unnormalized form at critical points. The unit Green flux gives
\begin{equation}\label{eq:introflux}
 s_\#(h^2\mu)=H^2\,\mathrm ds,\qquad
 s_\#(\kappa\mu)=c\,\mathrm ds,\qquad
 s_\#\mu\le C_G^2\,\mathrm ds.
\end{equation}
For $3\le n\le6$, one coefficient choice controls all three space forms. Its algebraic lower bound is proved at the least favorable hyperbolic parameter and then transferred to $c\ge0$ by an exact identity retaining the nonnegative ambient-curvature terms. Surfaces admit a shorter three-term combination. The resulting cutoff inequalities contradict the nondecaying $H^2$ or $c$ flux along long logarithmic intervals. In the hyperbolic case the negative ambient terms are absorbed before integration; its signed $c\,\mathrm ds$ flux is not used as a positive measure.

\subsection{Relation to earlier work}
The finite-index CMC problem is closely related to do Carmo's question. The surface theory includes da Silveira~\cite{Sil87}; higher-dimensional curvature and radius estimates appear in Cheng~\cite{Cheng06} and Elbert--Nelli--Rosenberg~\cite{ENR07}. Chen--Hong--Li~\cite{CHL25} prove minimality for complete noncompact finite-index CMC hypersurfaces in $\mathbb R^6$, without a volume-growth assumption. Miranda~\cite[Theorems~A and~D]{Mir26} treats ambient dimension six in nonnegative-curvature products and in hyperbolic space. Together with the minimal-hypersurface results of Catino--Mastrolia--Roncoroni~\cite{CMR24}, his work gives finite-index CMC compactness in $\mathbb S^6$. Thus the Euclidean and spherical conclusions in intrinsic dimensions below six overlap the earlier theory. The present argument treats these conclusions uniformly and includes the upper intrinsic dimension six. The compact weakly stable spherical classification used in Corollary~\ref{cor:weak} is the theorem of Barbosa--do Carmo--Eschenburg~\cite{BCE88}.

For hyperbolic hypersurfaces, Deng~\cite{Deng11} proves compactness under $H^2>64/63$ when $n=3$ and under $H^2>175/148$ when $n=4$. These normalized thresholds are also recorded in \cite[Appendix~A]{Mir26}. Miranda's condition $|\tr A|>7$ in $\mathbb H^6$ is $H^2>49/25$ when $n=5$. The corresponding sufficient constants in \eqref{eq:threshold} are smaller:
\[
 \frac{251}{250}<\frac{64}{63},\qquad
 \frac{207}{200}<\frac{175}{148},\qquad
 \frac{29}{25}<\frac{49}{25}.
\]
Hong~\cite{Hong25} obtains the condition $|H|>1$ in $\mathbb H^4$ with the additional hypothesis of finite topology. No finite-topology assumption is imposed here. The constants in \eqref{eq:threshold} arise from the explicit coefficient rows of Section~\ref{sec:algebra}, and are not asserted to be optimal for $n\ge3$.

The geometric reduction uses the stable Bernstein theorem of Hong--Li--Wang~\cite{HLW26}, continuing the work of do Carmo--Peng~\cite{DP79}, Fischer-Colbrie--Schoen~\cite{FCS80}, Chodosh--Li~\cite{CL24}, Chodosh--Li--Minter--Stryker~\cite{CLMS26}, and Mazet~\cite{Maz24}. The required variable-ambient CMC compactness is proved in Section~\ref{subsec:compactness}; its fixed-ambient counterpart is \cite[Proposition~2.9]{Mir26}. The analytic argument is related to the Green-function methods of \cite{CCMMR26,HLW26}. Here positive ambient curvature or nonzero mean curvature supplies a nondecaying flux, so a compactly supported estimate at Green exponent two suffices. In contrast, dimension-independent results based on volume growth, such as \cite{AdC93,INS16}, retain additional growth hypotheses. Section~\ref{subsec:entropy} records the classical zero-entropy consequence for comparison. All coefficient identities and positivity estimates needed for the main proof are stated explicitly in the paper.

\subsection{Organization}
Section~\ref{sec:geometry} fixes the conventions and proves the finite-index curvature reduction. Section~\ref{sec:green} constructs the Green function on a stable exterior. Sections~\ref{sec:identities} and~\ref{sec:algebra} derive the eight compactly supported identities and a common coercive combination. Section~\ref{sec:cutoff} applies logarithmic cutoffs and completes the main theorems and their corollaries. Section~\ref{sec:examples} gives stable hyperbolic tubes and explains the dimensional scope of the method. Appendix~\ref{app:coefficients} contains the exact coefficient arithmetic; Appendix~\ref{app:entropy} supplies a direct proof of the recalled zero-entropy result.

\section{Geometric preliminaries and curvature reduction}\label{sec:geometry}
\subsection{Conventions and space-form identities}\label{subsec:spaceform}
All scalar stability tests are real-valued. All distances and balls are intrinsic. We write $d_g$ for distance, $B_r(x)$ for an open ball, and $\mathrm dV_g$ for the $n$-dimensional volume measure. The measure $\mathrm dA_g$ on a regular level hypersurface is the induced $(n-1)$-dimensional area measure. All measures are taken on the immersed domain, including multiplicity when different sheets have the same image.

The connection $\nabla$ is the Levi--Civita connection of $g$, and $\nabla^2u$ is the covariant Hessian. For symmetric covariant two-tensors $T$ and $S$, we use the metric to identify them with self-adjoint endomorphisms. In a local orthonormal frame $e_1,\ldots,e_n$, our conventions are
\[
 \langle T,S\rangle=\sum_{i,j=1}^nT_{ij}S_{ij},\qquad
 |T|^2=\langle T,T\rangle,\qquad
 (T^2)_{ij}=\sum_{k=1}^nT_{ik}T_{kj}.
\]
Thus $T^2$ is an endomorphism square, not an entrywise square, and $\tr(T^k)$ is the trace of the endomorphism power. We set
\[
 (\divg T)_j=\sum_{i=1}^n(\nabla_{e_i}T)(e_i,e_j),\qquad
 |\nabla T|^2=\sum_{i,j,k=1}^n
       |(\nabla_{e_i}T)(e_j,e_k)|^2.
\]
When endomorphisms are used, $g$ denotes the identity and $v\otimes v$ denotes the rank-one map $X\mapsto\langle v,X\rangle v$; in covariant notation it is $v^\flat\otimes v^\flat$. A scalar cutoff appearing under an integral is always composed with its stated scalar parameter.

The sign convention is \eqref{eq:conventions}, and $a=|B|$. Since $H$ is constant and the ambient curvature is constant, $B$ is a trace-free Codazzi tensor and $\nabla B$ is a fully symmetric trace-free covariant three-tensor. Consequently
\begin{equation}\label{eq:divergences}
 \divg B=0,\qquad
 \divg\left(\tfrac12a^2g-B^2\right)=0,\qquad
 \divg(HB)=0.
\end{equation}
For the second identity, in a geodesic orthonormal frame, Codazzi gives
\[
 \nabla_i(B_{ik}B_{kj})=B_{ik}\nabla_iB_{kj}
 =B_{ik}\nabla_jB_{ki}=\tfrac12\nabla_ja^2.
\]
The Gauss equation is
\begin{equation}\label{eq:Ric}
 \Ric=(n-1)c\,g+nHA-A^2
     =(n-1)(c+H^2)g+(n-2)HB-B^2.
\end{equation}
Commuting derivatives in the Codazzi equation gives
\begin{equation}\label{eq:simonsA}
 \Delta A=nHA^2-|A|^2A+ncA-ncH g.
\end{equation}
Indeed, Codazzi and $\divg A=0$ give
\[
 (\Delta A)_{ij}=\sum_k([\nabla_k,\nabla_i]A)_{kj}.
\]
For covariant tensors,
$([\nabla_X,\nabla_Y]A)(Z,W)=-A(R(X,Y)Z,W)-A(Z,R(X,Y)W)$.
Substitution of the Gauss equation
\[
 R(X,Y)Z=c(\langle Y,Z\rangle X-\langle X,Z\rangle Y)
       +A(Y,Z)AX-A(X,Z)AY
\]
gives exactly \eqref{eq:simonsA}. See also
\cite{Sim68,SY94}. Since $\Delta B=\Delta A$,
\[
 \tr(BA^2)=\tr B^3+2Ha^2,\qquad
 \tr(BA)=a^2,\qquad \tr B=0
\]
give the scalar CMC Simons identity
\begin{equation}\label{eq:simons}
 \tfrac12\Delta a^2=|\nabla B|^2+n(c+H^2)a^2-a^4+nH\tr B^3.
\end{equation}
Stability on an open set $\mathcal O\subset M$ means
\begin{equation}\label{eq:stability}
 \int_M\bigl(a^2+n(c+H^2)\bigr)f^2\dV
 \le\int_M|\nabla f|^2\dV,
 \qquad f\in C_c^\infty(\mathcal O).
\end{equation}
Local approximation extends this inequality to compactly supported $W^{1,2}$ functions in $\mathcal O$.

\subsection{Exterior stability and the rigidity input}\label{subsec:exterior}
The following disjoint-support argument is the standard finite-index localization; see \cite{FC85}.
\begin{lemma}\label{lem:exterior}
If $\ind(Q_c)<\infty$, there is a compact set $K\subset M$ such that $M\setminus K$ is strongly stable. Moreover,
\begin{equation}\label{eq:indexcomparison}
 \ind_0(Q_c)\le\ind(Q_c)\le\ind_0(Q_c)+1.
\end{equation}
\end{lemma}
\begin{proof}
If no such compact set existed, one could choose $f_j\in C_c^\infty(M)$ with pairwise disjoint supports and $Q_c(f_j)<0$. At each step, choose the next negative test outside a compact neighborhood of the previous supports. All mixed terms vanish, so the form is negative definite on the span of any finite subcollection. This contradicts finite index.

The first inequality in \eqref{eq:indexcomparison} follows by restriction. For the second, let $E\subset C_c^\infty(M)$ be any finite-dimensional negative-definite subspace. Since the zero-mean space \eqref{eq:zeroMean} has codimension at most one,
\[
 \dim(E\cap\mathcal V_0)\ge\dim E-1.
\]
Taking suprema over $E$ gives the result, including the extended-valued case. In particular, finite constrained index implies finite strong index.
\end{proof}

\begin{theorem}[Hong--Li--Wang, with its product consequence]\label{thm:HLW}
For $2\le n\le6$, every smooth, connected, complete, boundaryless, two-sided
stable minimal immersion $M^n\to\mathbb R^{n+1}$ is totally geodesic and has
image an affine $n$-plane.
\end{theorem}
\begin{proof}[Derivation from the six-dimensional theorem]
The case $n=6$ is \cite[Theorem~1.1]{HLW26}. For $n<6$ use the product immersion
$\widetilde F(x,y)=(F(x),y)$ of $M\times\mathbb R^{6-n}$ into $\mathbb R^7$.
It is smooth, connected, complete, two-sided and minimal, with
$|\widetilde A|(x,y)=|A|(x)$. For any compactly supported smooth $f(x,y)$,
Fubini and stability on each $M$-slice give
\[
 \int_{M\times\mathbb R^{6-n}}|\widetilde A|^2f^2
 \le\int_{M\times\mathbb R^{6-n}}|\nabla_M f|^2
 \le\int_{M\times\mathbb R^{6-n}}|\nabla f|^2.
\]
The six-dimensional theorem makes this product flat, hence $A=0$ on $M$.
Completeness turns the immersion into a covering of its affine plane; the
plane is simply connected. This argument is also recorded as
\cite[Corollary~1.2]{HLW26}. It is applied to globally stable minimal limits,
not to a product of a merely finite-index nonstable immersion.
\end{proof}

\subsection{Compactness for rescaled CMC immersions}\label{subsec:compactness}
For this subsection only, extend the notation by setting
$X_d^{n+1}$ to be the simply connected complete $(n+1)$-dimensional space form of curvature $d$, for any $d\in\mathbb R$.
\begin{proposition}[CMC compactness with a flat ambient limit]\label{prop:compactness}
Let $c_j\in\mathbb R$ with $c_j\to0$, and let
$F_j:(N_j^n,g_j)\to X_{c_j}^{n+1}$ be smooth two-sided CMC immersions of compact connected manifolds with smooth nonempty boundary. Suppose that points $p_j\in N_j$ satisfy
\begin{equation}\label{eq:compactnesshyp}
 d_{g_j}(p_j,\partial N_j)\longrightarrow\infty,
 \qquad \sup_j\sup_{N_j}|A_j|<\infty.
\end{equation}
After ambient isometries and passage to a subsequence, the immersions converge smoothly in pointed intrinsic charts to a complete, connected, noncompact, boundaryless, two-sided CMC immersion into $\mathbb R^{n+1}$. The mean curvatures converge. If the approximating immersions are strongly stable, then the limit is strongly stable.
\end{proposition}
The fixed Euclidean case is \cite[Proposition~2.9]{Mir26}. We give the local compactness argument, including the changing space-form metrics, in the form needed here. The geometric chart construction is the intrinsic one in \cite[Lecture~3, proof of Theorem~22]{Whi16}; smooth Riemannian compactness is also described in \cite{HH97}.
\begin{proof}
Write $\Lambda_0=\sup_j\sup_{N_j}|A_j|$. Since $|H_j|\le\Lambda_0/\sqrt n$, a subsequence satisfies $H_j\to H_\infty$.

\emph{Ambient normalization.} Apply ambient isometries to place $F_j(p_j)$ at the origin of normal coordinates. The polar-coordinate form of the space-form metric is
\[
 \bar g_{c_j}=\mathrm dr^2+\operatorname{sn}_{c_j}(r)^2g_{\mathbb S^n},\qquad
 \operatorname{sn}_d(r)=
 \begin{cases}
 \sin(\sqrt d\,r)/\sqrt d,&d>0,\\
 r,&d=0,\\
 \sinh(\sqrt{-d}\,r)/\sqrt{-d},&d<0.
 \end{cases}
\]
In Cartesian normal coordinates these metrics converge in $C^\infty$ on every fixed ball to the Euclidean metric, from either sign of $c_j$. For $c_j>0$ the normal-coordinate radius tends to infinity. Length preservation places the image of each fixed intrinsic ball in the corresponding ambient ball. This treats negative curvature without using the indefinite ambient metric of the hyperboloid model.

\emph{Uniform interior charts.} Gauss gives a uniform intrinsic sectional-curvature bound. At points whose distance from the boundary is bounded below, there is also a uniform positive interior injectivity radius. All radii here are smaller than a
fixed fraction of the distance to the boundary; thus this is an interior
statement, not a claim about injectivity at the boundary. Here are the two ingredients of this familiar local estimate. Comparison gives a positive conjugate-radius bound. If the injectivity radius nevertheless tends to zero, the cut-point alternative produces geodesic loops of lengths tending to zero, with at most one corner. Their images lie in ambient normal-coordinate balls of radii tending to zero. In these coordinates, the metrics and Christoffel symbols are uniformly bounded, while the immersion equation bounds the Euclidean curvature of the images by a uniform constant in terms of $\Lambda_0$ and the ambient geometry. Thus their total Euclidean curvature, apart from the corner, tends to zero. The corner contributes at most $\pi$, contradicting Fenchel's total-curvature bound $2\pi$ for a closed
piecewise smooth curve. More explicitly, if its smooth arcs have Euclidean
curvature at most $C$ and total length $\ell$, then the total curvature is
at most $C\ell+\pi<2\pi$ when $\ell<\pi/C$. The cut-point alternative
is used only below the conjugate radius and below the distance to the
boundary (\cite{Pet16}). The entire argument lies away from the boundary. This is the local immersed chart mechanism used in \cite[Lecture~3, proof of Theorem~22]{Whi16} and \cite[Proposition~2.9]{Mir26}; it concerns a distinguished intrinsic sheet, not the total area of all sheets in an extrinsic ball.

Bounded second fundamental form controls the rotation of tangent planes along short intrinsic curves. Projection of the distinguished local sheet onto its tangent $n$-plane is therefore uniformly nonsingular at a fixed scale. Lifting radial segments constructs a graph on a uniform disk, with bounded slope and second derivatives. After decreasing the disk radius, the graph contains a uniform intrinsic ball. The boundary-distance condition prevents the lifts from leaving the domain. Other sheets passing through the same ambient region are not identified with this sheet.

\emph{Elliptic estimates.} In these charts the CMC equation is a uniformly elliptic second-order graph equation with constant right-hand side $nH_j$ and coefficients determined by the ambient metric. In the Euclidean case it reads
\[
 -\divg_{\mathbb R^n}\left(\frac{Du_j}{\sqrt{1+|Du_j|^2}}\right)=nH_j
\]
for the upward graph normal; reversing the graph orientation reverses both signs. The ambient normal-coordinate metrics have uniform bounds of every order on fixed balls. The uniform slope and second-derivative bounds give uniformly H\"older coefficients on smaller disks. Interior Schauder estimates (\cite[Theorem~6.2]{GT01}), followed by differentiation of the equation, yield uniform $C^{k,\alpha}$ estimates for every $k$ and $0<\alpha<1$. The induced metrics and transition maps consequently satisfy uniform local smooth estimates.

\emph{The pointed limit.} On each fixed intrinsic ball, the sectional-curvature and interior injectivity bounds give a uniformly finite cover by the preceding charts; this follows from volume comparison and the fixed lower volume of smaller chart balls. Smooth pointed compactness and a diagonal subsequence produce a connected limit $(N_\infty,g_\infty,p_\infty)$ and a smooth isometric immersion $F_\infty$. Its image lies in the limiting Euclidean space $\mathbb R^{n+1}$. The distance to the boundary tends to infinity, so the limit is complete and has no boundary. A sequence of minimizing paths from $p_j$ to $\partial N_j$, restricted to successively longer initial segments, converges to a minimizing ray in $N_\infty$; in particular the limit is noncompact.

Choose the same orientation on overlapping intrinsic charts before taking
the subsequence. The equations $\bar\nabla N_j=A_j$ and the graph estimates
give smooth convergence of the normals. Since the normals agree on overlaps
for each $j$, their limits also agree, giving a global normal. This does not
identify distinct sheets that have the same ambient image. The CMC equations pass to the limit with mean curvature $H_\infty$. For each $f\in C_c^\infty(N_\infty)$, use the pointed identifications on a neighborhood of its support to transfer it to $N_j$ and then extend by zero; the support lies strictly inside that neighborhood, so the extension is smooth. Smooth convergence and $c_j\to0$ give
\[
 \int_{N_\infty}(|\nabla f|^2-|A_\infty|^2f^2)\,\mathrm dV_{g_\infty}
 =\lim_{j\to\infty}\int_{N_j}
 \bigl(|\nabla f_j|^2-(|A_j|^2+nc_j)f_j^2\bigr)\,\mathrm dV_{g_j}.
\]
If each $F_j$ is strongly stable, this limit is nonnegative. This proves all assertions.
\end{proof}

\subsection{The global curvature bound}\label{subsec:curvature}
\begin{lemma}[Bounded curvature at finite index]\label{lem:curvature}
Assume Theorem~\ref{thm:HLW}. Every complete two-sided CMC immersion $F:M^n\to X_c^{n+1}$, $c\in\{-1,0,1\}$, of finite strong Morse index has bounded second fundamental form.
\end{lemma}
\begin{proof}
For a compact domain the assertion is immediate. Otherwise choose a nonempty compact $K$ from Lemma~\ref{lem:exterior}. If $|A|$ were unbounded, smoothness on compact sets would give points $p_j$ such that
\[
 m_j:=|A|(p_j)\longrightarrow\infty,
 \qquad d_g(p_j,K)>2.
\]
On the compact ball $\overline{B_1(p_j)}$, maximize
\[
 x\longmapsto |A|(x)(1-d_g(p_j,x)).
\]
Let $x_j$ be an interior maximum point, put $\lambda_j=|A|(x_j)$, and set
\[
 r_j=\tfrac12(1-d_g(p_j,x_j))>0.
\]
Maximality gives
\begin{equation}\label{eq:pointpicking}
 \lambda_j\ge m_j,\qquad
 \lambda_j r_j\ge\tfrac12m_j,\qquad
 |A|\le2\lambda_j\quad\text{on }B_{r_j}(x_j).
\end{equation}
Indeed, $1-d_g(p_j,x)\ge r_j$ on this smaller ball. These balls lie in $M\setminus K$.

Choose smooth connected relatively compact domains with
\[
 \overline{B_{r_j/2}(x_j)}\subset\Omega_j,
 \qquad \overline{\Omega_j}\subset B_{r_j}(x_j).
\]
Smooth neighborhood approximation of the compact inner ball provides them; take the connected component containing that ball. Their boundaries are nonempty. Multiply the ambient and induced metrics by $\lambda_j^2$ and apply ambient isometries at $F(x_j)$. On $N_j=\overline{\Omega_j}$ the rescaled data satisfy
\begin{equation}\label{eq:blowupdata}
\begin{aligned}
 |\widetilde A_j|&\le2,&
 |\widetilde A_j|(x_j)&=1,&
 \widetilde H_j&=H/\lambda_j\longrightarrow0,\\
 \widetilde c_j&=c/\lambda_j^2\longrightarrow0,&
 d_{\lambda_j^2g}(x_j,\partial N_j)&\ge\lambda_jr_j/2\longrightarrow\infty.
\end{aligned}
\end{equation}
Strong stability is preserved: in intrinsic dimension $n$,
\begin{equation}\label{eq:scaledQ}
 \widetilde Q_j(f)=\lambda_j^{n-2} Q_c(f)
\end{equation}
for tests on $\Omega_j$. Proposition~\ref{prop:compactness} therefore gives a complete connected two-sided stable minimal immersion into $\mathbb R^{n+1}$ with $|A_\infty|(p_\infty)=1$. This contradicts Theorem~\ref{thm:HLW}. Hence $|A|$ is globally bounded.
\end{proof}

\section{Green functions on a stable exterior}\label{sec:green}
In Sections~\ref{sec:green}--\ref{sec:cutoff} we prove Theorem~\ref{thm:bounded}. Until its proof is completed, assume that $M$ is noncompact, that $\Lambda:=\sup_M|A|<\infty$, and that \eqref{eq:stability} holds with $\mathcal O=M\setminus K$. Enlarge $K$ to a nonempty compact set. We argue by contradiction under
\begin{equation}\label{eq:positiveGamma}
 \gamma:=c+H^2>0,\qquad\text{and, when }c=-1,\quad\text{condition \eqref{eq:threshold}.}
\end{equation}
Reversing $N$ if necessary, take $H\ge0$; the case $H=0$ is retained when $c=1$. Equation~\eqref{eq:Ric} gives
\begin{equation}\label{eq:Riclower}
 \Ric\ge-\bigl((n-1)\max\{-c,0\}+nH\Lambda+\Lambda^2\bigr)g.
\end{equation}
Exterior stability gives the Dirichlet spectral lower bound $n\gamma$ outside $K$. A separate argument is needed to obtain a strictly positive global spectral bottom.

\subsection{Intrinsic volume and the global spectral bottom}
\begin{lemma}[Small-ball lower bound]\label{lem:volume}
Set $\Theta=\sqrt{H^2+\max\{c,0\}}$. For every $x\in M$ and $r>0$,
\begin{equation}\label{eq:volume}
 \Vol_g B_r(x)\ge\omega_n r^ne^{-n\Theta r},
 \qquad \omega_n=\frac{\pi^{n/2}}{\Gamma(n/2+1)}.
\end{equation}
In particular, the noncompact manifold $M$ has infinite intrinsic volume.
\end{lemma}
\begin{proof}
For $c=0$, use the Euclidean immersion; for $c=1$, compose with $\mathbb S^{n+1}(1)\subset\mathbb R^{n+2}$. Respectively,
\[
 \Delta F=-nHN,\qquad \Delta F=-nF-nHN.
\]
In the latter formula the summands are orthogonal, so in both cases $|\Delta F|=n\Theta$. For $u(y)=|F(y)-F(x)|^2$, length preservation gives, on $B_r(x)$,
\[
 |\nabla u|\le2r,\qquad \Delta u\ge2n-2n\Theta r.
\]

For $c=-1$, instead let $\rho$ be ambient distance from $F(x)$ and set $u=\rho^2\circ F$. The squared distance is smooth on the simply connected hyperbolic space, including at its center. Away from that center,
\[
 \bar\nabla^2\rho=\coth\rho\,(\bar g-\mathrm d\rho\otimes\mathrm d\rho),
 \qquad \bar\nabla^2(\rho^2)\ge2\bar g,
\]
since $\rho\coth\rho\ge1$. The trace formula for an immersed function is
\[
 \Delta_M u=\sum_{i=1}^n\bar\nabla^2(\rho^2)(dF(e_i),dF(e_i))
       -nH\langle\bar\nabla(\rho^2),N\rangle.
\]
Because $\rho\circ F\le d_g(x,\cdot)$, this gives $\Delta_Mu\ge2n-2n|H|r$ and $|\nabla_Mu|\le2r$ on $B_r(x)$. These are the same inequalities with $\Theta=|H|$. No positive-definite norm on a Lorentzian ambient space is used.

Writing $V(r)=\Vol_g B_r(x)$, the divergence theorem on distance balls and coarea imply, for almost every $r>0$,
\[
 (2n-2n\Theta r)V(r)\le\int_{B_r(x)}\Delta u\dV\le2rV'(r).
\]
This is the Gauss--Green formula for almost every finite-perimeter distance ball and the coarea formula; equivalently one may use Lipschitz radial approximation (\cite[Chapters~3 and~5]{EG15}). Integrating $(\log V)'\ge n/r-n\Theta$ and using $V(\varepsilon)/(\omega_n\varepsilon^n)\to1$ proves \eqref{eq:volume}. Noncompactness and completeness give infinitely many disjoint fixed-radius intrinsic balls. Their uniform positive volume proves the final assertion.
\end{proof}

Define
\begin{equation}\label{eq:lambda0}
 \lambda_0=\inf_{0\ne f\in C_c^\infty(M)}
 \frac{\int_M|\nabla f|^2\dV}{\int_M f^2\dV}.
\end{equation}
\begin{lemma}[From an exterior gap to a global positive spectrum]\label{lem:spectrum}
Under the standing assumptions, $\lambda_0>0$.
\end{lemma}
\begin{proof}
If $\lambda_0=0$, choose $u_j\in C_c^\infty(M)$ with $\|u_j\|_2=1$ and $\|\nabla u_j\|_2\to0$. Rellich compactness on relatively compact coordinate domains (\cite[Chapter~7]{GT01}), followed by a diagonal subsequence, give convergence in $L^2_{\mathrm{loc}}$ to a function with zero weak gradient. Connectedness makes the limit a constant. Fatou's lemma and $\Vol_g(M)=\infty$ force that constant to be zero.

Choose $0\le\zeta_0\le1$ smooth, zero on a neighborhood of $K$, and equal to one outside a larger compact set. Exterior stability yields
\[
 n\gamma\int_M\zeta_0^2u_j^2\dV
 \le\int_M|\nabla(\zeta_0 u_j)|^2\dV
 \le2\int_M|\nabla u_j|^2\dV
       +2\int_Mu_j^2|\nabla\zeta_0|^2\dV\longrightarrow0.
\]
The remaining $L^2$ mass also tends to zero because it lies in a fixed compact set. This contradicts $\|u_j\|_2=1$.
\end{proof}
\begin{remark}\label{rem:lambda0}
The exterior lower bound $n\gamma$ is not asserted to hold globally. From now on the symbol $\lambda_0$ always denotes the positive number in \eqref{eq:lambda0}, not $n\gamma$. This distinction is essential when only exterior stability is assumed.
\end{remark}

\subsection{A unit-pole Green function and its level sets}
\begin{lemma}\label{lem:green}
Fix $p\in M$. There is a minimal positive Green function $G$ with
\begin{equation}\label{eq:green}
 -\Delta G=\delta_p,\qquad
 G(x)\longrightarrow0\quad\text{as }d_g(p,x)\longrightarrow\infty.
\end{equation}
For every $t>0$,
\begin{equation}\label{eq:truncation}
 \lambda_0\int_M\min\{G,t\}^2\dV\le t.
\end{equation}
Every positive finite Green slab is a compact subset of $M\setminus\{p\}$, and for every positive regular value $t$,
\begin{equation}\label{eq:unitflux}
 \int_{\{G=t\}}|\nabla G|\dAg=1.
\end{equation}
There are $t_0>0$ and $C_G<\infty$ such that
\begin{equation}\label{eq:gradient}
 \mathcal O_G:=\{0<G<t_0\}\subset M\setminus K,
 \qquad \frac{|\nabla G|}{G}\le C_G\quad\text{on }\mathcal O_G.
\end{equation}
\end{lemma}
\begin{proof}
Choose a nested exhaustion by smooth connected relatively compact domains $\Omega_j$ containing $p$. Let $G_j$ be the unit-pole Dirichlet Green function on $\Omega_j$; see \cite[Chapter~17]{Li12} and \cite{Gri09}. The maximum principle gives $G_j\le G_{j+1}$. The truncation $v_{j,t}=\min\{G_j,t\}$, extended by zero, belongs to $W_0^{1,2}(\Omega_j)$ and is constant equal to $t$ near $p$. Testing the Green equation gives
\begin{equation}\label{eq:truncatedenergy}
 \int_{\Omega_j}|\nabla v_{j,t}|^2\dV=t.
\end{equation}
For this test, take smooth approximants that remain identically $t$ on a fixed smaller pole neighborhood while approximating on its complement. Their gradients converge in $L^2$, and their values at $p$ remain $t$. Thus the right-hand side is exactly $t$; no continuity of point evaluation on $W^{1,2}$ is used. Applying \eqref{eq:lambda0} by density yields
\begin{equation}\label{eq:truncj}
 \lambda_0\int_M\min\{G_j,t\}^2\dV\le t.
\end{equation}

Fix $x\ne p$ and a small ball $B_{2r}(x)$ disjoint from $p$. For all large $j$, Harnack's inequality gives $G_j\ge c_xG_j(x)$ on $B_r(x)$, with $0<c_x\le1$ independent of $j$. Taking $t=G_j(x)$ in \eqref{eq:truncj} gives
\[
 \lambda_0c_x^2G_j(x)^2\Vol_g B_r(x)\le G_j(x),
\]
so $G_j(x)$ is uniformly bounded. Harnack's inequality and interior elliptic estimates give smooth convergence off $p$ to a positive harmonic function $G$. To retain the unit pole, fix a small smooth domain $D$ about $p$ and subtract its Dirichlet Green function $G_D$. The functions $G_j-G_D$ are harmonic across $p$, with bounded convergent boundary values on $\partial D$. Their harmonic limit shows that $-\Delta G=\delta_p$.
Writing $r=d_g(p,\cdot)$ and $\sigma_{n-1}=|\mathbb S^{n-1}|$, the local
Dirichlet singularity is $G_D(r)\sim[(n-2)\sigma_{n-1}]^{-1}r^{2-n}$ for
$n\ge3$, and $G_D(r)=-(2\pi)^{-1}\log r+O(1)$ for $n=2$.
The harmonic correction is bounded near the pole, so in both cases
$G\to\infty$ at $p$. The exhaustion construction gives minimality. Fatou's lemma in \eqref{eq:truncj} proves \eqref{eq:truncation}.

Choose a constant $K_{\mathrm{Ric}}\ge0$ such that $\Ric\ge-(n-1)K_{\mathrm{Ric}}^2g$, using \eqref{eq:Riclower}. The local Cheng--Yau estimate \cite{CY75} gives
\begin{equation}\label{eq:CY}
 \sup_{B_r(x)}|\nabla\log v|
 \le C_n(r^{-1}+K_{\mathrm{Ric}})
\end{equation}
for positive harmonic $v$ on $B_{2r}(x)$. Therefore fixed-radius balls avoiding the pole have a uniform Harnack constant for $G$.

If $G$ did not tend to zero at infinity, there would be $\varepsilon>0$ and pairwise disjoint balls $B_1(x_j)$ escaping the pole such that $G(x_j)\ge\varepsilon$. Equation~\eqref{eq:CY} implies $G\ge c_*\varepsilon$ on each ball, with $c_*>0$ independent of $j$. Lemma~\ref{lem:volume} then implies
\[
 \int_M\min\{G,c_*\varepsilon\}^2\dV=\infty,
\]
contrary to \eqref{eq:truncation}. This proves decay.

Completeness, decay at infinity, and blow-up at the pole imply that $\{\alpha\le G\le\beta\}$ is compact for $0<\alpha<\beta<\infty$. For a positive regular value $t$, the open superlevel set with its pole included is relatively compact; its closure is a compact region with smooth boundary. Integrating the unit-pole equation over that region gives \eqref{eq:unitflux}, because the outward conormal is $-\nabla G/|\nabla G|$. Sard's theorem makes the formula available for almost every $t>0$.

Finally enlarge $K$ to contain $\overline{B_3(p)}$. Positivity away from $p$ and pole blow-up give a positive infimum of $G$ on $K\setminus\{p\}$. Choose $t_0$ below this infimum. Then $\mathcal O_G\subset M\setminus K$, and \eqref{eq:CY} on unit-scale balls gives the stated logarithmic-gradient bound.
\end{proof}

\section{Compactly supported Green identities}\label{sec:identities}
All test functions in this section are supported in a positive finite Green slab contained in the strongly stable region $\mathcal O_G=\{0<G<t_0\}$. No integrability at infinity is assumed.

\subsection{Normalized variables and critical sets}
On $\mathcal O_G\cap\{|\nabla G|>0\}$, put
\begin{equation}\label{eq:variables}
\begin{aligned}
 w&=\frac{|\nabla G|}{G},&
 \nu&=\frac{\nabla G}{|\nabla G|},&
 S&=\frac Bw,&b&=|S|,&h&=\frac Hw,\\
 \mathcal C&=\frac{\nabla B}{w^2},&
 Y&=\frac{\nabla a}{w^2},&
 U&=|S\nu|^2,&V&=\ip{S\nu}{\nu},\\
 Z&=\frac{\nabla^2G}{Gw^2}-\frac{n\nu\otimes\nu-g}{n-1},&&
 \mathrm d\mu=w^4G\,\mathrm dV_g.
\end{aligned}
\end{equation}
We also introduce the signed normalized ambient-curvature scalar
\begin{equation}\label{eq:kappa}
 \kappa=\frac{c}{w^2}.
\end{equation}
Thus $\kappa=0$ in the Euclidean case, $\kappa=w^{-2}$ in the unit sphere, and $\kappa=-w^{-2}$ in hyperbolic space. It is a scalar function, not a constant curvature bound or a tensor.

Here $w$ is a nonnegative scalar logarithmic-gradient norm; $S$ and $Z$ are dimensionless self-adjoint two-tensors; $b$ and $h$ are scalars; $\mathcal C$ is a covariant three-tensor; $Y$ and $\nu$ are tangent vectors; and $U,V$ are scalar contractions. In particular, $a=|B|$ is the norm of the trace-free second fundamental form, not $|A|$, and
\[
 b=\frac a w,\qquad h=\frac H w,\qquad
 \mathcal C_{ijk}=w^{-2}(\nabla_{e_i}B)(e_j,e_k).
\]
The unnormalized shifted Hessian associated with $Z$ is
\begin{equation}\label{eq:unnormalizedZ}
 \mathscr H_G=\nabla^2G-\frac1{(n-1)G}
       \bigl(n\,\mathrm dG\otimes\mathrm dG-|\nabla G|^2g\bigr),
 \qquad Z=\frac{\mathscr H_G}{Gw^2}.
\end{equation}
The symbol $Z$ in the Green identities denotes the dimensionless tensor in \eqref{eq:variables}; $\mathscr H_G$ denotes its unnormalized counterpart.

Both $S$ and $Z$ are trace-free. The vector $\nu$ is tangent to $M$ and distinct from the normal $N$ to the immersion. A colon, as in $S^2:Z$, denotes full tensor contraction. For $0\le\varphi\in C_c^\infty((0,t_0))$, set
\begin{equation}\label{eq:cutoffnotation}
 D=t\frac{\mathrm d}{\mathrm dt},\qquad
 \omega_\varphi(t)=\varphi(t)^4,\qquad
 \langle T\rangle_\varphi=\int_{\mathcal O_G} T\varphi(G)^4\dmu.
\end{equation}
Functions of $t$, including their $D$-derivatives, are always differentiated before composition with $G$. For example, $D\omega_\varphi$ under an integral means $(D\omega_\varphi)(G(x))$. The notation $\langle T\rangle_\varphi$ will be used only for a scalar density $T=T(x)$, and includes the cutoff weight; it is not a pointwise tensor inner product. If $s=\log(t_0/G)$ and $\Phi(s)=\varphi(t_0e^{-s})$, then
\[
 (D\varphi)(G)=-\Phi'(s),\qquad
 (D^2\varphi)(G)=\Phi''(s).
\]
The cutoff weight $\omega_\varphi$ is distinct from the Euclidean ball-volume constant $\omega_n$.

To interpret the normalized expressions at critical points, note that CMC graphs in the analytic space forms $X_c^{n+1}$ satisfy an analytic locally uniformly elliptic equation and are therefore analytic by \cite{Mor58}. In the resulting graph charts, the induced metric and the harmonic function $G$ are analytic. Since $G$ is nonconstant, its critical set has volume zero. The norms $a$ and $w$ are locally Lipschitz, and their weak gradients are used. More precisely, choose regular values $0<a_0<a_1<t_0$ outside the Green-value
support of the test; its preimage is contained in the interior of the compact
smooth slab $\{a_0\le G\le a_1\}$. All three stability tests below belong to
$W_0^{1,2}$ of this surrounding slab. The cutoff support itself need not have
smooth boundary.

Weighted expressions are first interpreted in the original variables. For example,
\begin{equation}\label{eq:unnormalized}
\begin{aligned}
 b^4\dmu&=a^4G\dV,& h^4\dmu&=H^4G\dV,\\
 |\mathcal C|^2\dmu&=|\nabla B|^2G\dV,& |Y|^2\dmu&=|\nabla a|^2G\dV,\\
 h^2b^2\dmu&=H^2a^2G\dV.
\end{aligned}
\end{equation}
The additional curvature densities have the unnormalized interpretations
\begin{equation}\label{eq:kappadensities}
\begin{aligned}
 \kappa\dmu&=c\frac{|\nabla G|^2}{G}\dV,\qquad
 \kappa b^2\dmu=c a^2G\dV,\qquad
 \kappa h^2\dmu=cH^2G\dV.
\end{aligned}
\end{equation}
These formulas define the weighted quantities across $w=0$, without assigning a separate boundary term to the critical set. Moreover,
\[
 w^2Z=\frac{\nabla^2G}{G}
       -\frac1{n-1}\bigl(n\,\mathrm d\log G\otimes\mathrm d\log G-w^2g\bigr)
\]
is smooth off the pole. Thus $|Z|^2\,\mathrm d\mu$ is locally integrable. Terms containing $Z\nu$
are bounded by $|Z|$. The mixed terms satisfy
\[
 (S^2:Z)\dmu=(B^2:\mathscr H_G)\dV,\qquad
 (hS:Z)\dmu=(HB:\mathscr H_G)\dV.
\] These formulas bound all remaining contractions on each
fixed slab, even though the normalized variables need not be bounded near
$\{w=0\}$. Green-gradient identities will be derived from the smooth function $w^2$, not from a formal distributional expression $w\Delta w$.

The term \emph{Green exponent two} refers to the measure in \eqref{eq:variables}. For a positive regular value $t$, set
\[
 J(t)=\int_{\{G=t\}}|\nabla G|^3\dAg.
\]
The coarea formula gives, for every compactly supported Green-value cutoff,
\begin{equation}\label{eq:GreenExponent}
 \int_{\mathcal O_G}\varphi(G)^4\dmu
 =\int_0^{t_0}\varphi(t)^4J(t)t^{-3}\dd t.
\end{equation}
Thus the weight $J(t)t^{-\sigma-1}\dd t$ has $\sigma=2$. This notation specifies the compact-test weight and asserts no uncut global moment is finite.

\subsection{The identities and exact cutoff errors}
Define
\begin{align}
 E_1&=|Z|^2-U+(n-2)hV+(n-1)(h^2+\kappa)+\frac1{n-1},\label{eq:E1}\\
 E_2&=S^2:Z+\frac n{n-1}U-\frac{b^2}{n-1},\label{eq:E2}\\
 E_3&=b^4-n(h^2+\kappa)b^2-nh\tr S^3-|\mathcal C|^2,\label{eq:E3}\\
 E_4&=|\mathcal C|^2-|Y|^2+2n(h^2+\kappa)b^2+nh\tr S^3-\tfrac14b^2,\label{eq:E4}\\
 E_5&=b^2+n(h^2+\kappa)-|Z\nu|^2-\tfrac14,\label{eq:E5}\\
 E_6&=Z(\nu,\nu),\label{eq:E6}\\
 E_7&=hS:Z+\frac n{n-1}hV,\label{eq:E7}\\
 E_8&=h^2\bigl(b^2+n(h^2+\kappa)-\tfrac14\bigr).\label{eq:E8}
\end{align}
\begin{proposition}\label{prop:identities}
With this notation,
\begin{equation}\label{eq:eightrelations}
 \langle E_i\rangle_\varphi=R_i\quad(i=1,2,3,6,7),\qquad
 \langle E_i\rangle_\varphi\le R_i\quad(i=4,5,8),
\end{equation}
where
\begin{align}
 R_1&=\tfrac12\int\bigl(D^2\omega_\varphi-\frac{n-3}{n-1}D\omega_\varphi\bigr)\dmu,\label{eq:R1}\\
 R_2&=\int(\tfrac12b^2-U)D\omega_\varphi\dmu,\label{eq:R2}\\
 R_3&=-\tfrac12\int b^2(D^2\omega_\varphi+D\omega_\varphi)\dmu,\label{eq:R3}\\
 R_4&=\int b^2\bigl(2\varphi^3D\varphi+4\varphi^2(D\varphi)^2\bigr)\dmu,\label{eq:R4}\\
 R_5&=-2\int\bigl(\varphi^2(D\varphi)^2+\varphi^3D^2\varphi\bigr)\dmu,\label{eq:R5}\\
 R_6&=-\tfrac12\int D\omega_\varphi\dmu,\label{eq:R6}\\
 R_7&=-\int hV D\omega_\varphi\dmu,\label{eq:R7}\\
 R_8&=\int h^2\bigl(2\varphi^3D\varphi+4\varphi^2(D\varphi)^2\bigr)\dmu.\label{eq:R8}
\end{align}
All integrals are over $\mathcal O_G$, and all scalar cutoff factors are composed with $G$. The $E_i=E_i(x)$ are pointwise scalar densities, whereas the $R_i=R_i[\varphi]$ are signed real numbers depending on the chosen cutoff.
\end{proposition}
\begin{proof}
We derive all eight relations with their cutoff terms. Ambient curvature enters only the interior densities $E_1,E_3,E_4,E_5,E_8$; the cutoff errors $R_i$ are independent of $c$ except through the geometric variables.

\emph{Bochner and the Green-gradient divergence.} Set $u=\log G$, so $\Delta u=-w^2$. On the regular set,
\begin{equation}\label{eq:gradw}
 \nabla w=w^2Z\nu.
\end{equation}
Indeed, $\nabla^2u=w^2[Z+(\nu\otimes\nu-g)/(n-1)]$, and contraction with $\nu$ cancels the last term. Since $\tr Z=0$,
\[
 |\nabla^2u|^2=w^4\left(|Z|^2+\frac2{n-1}Z(\nu,\nu)+\frac1{n-1}\right),
 \qquad \ip{\nabla u}{\nabla\Delta u}=-2w^4Z(\nu,\nu).
\]
Bochner's formula and \eqref{eq:Ric} therefore give
\begin{equation}\label{eq:bochner}
 \tfrac12\Delta w^2=w^4(E_1-\frac{2(n-2)}{n-1}E_6).
\end{equation}
For every smooth Green-value function $\eta$ supported in a slab,
\begin{equation}\label{eq:graddiv}
 \divg\bigl(w^2\eta(G)\nabla G\bigr)=Gw^4(2\eta E_6+D\eta).
\end{equation}
Integrating with $\eta=\omega_\varphi$ gives $\langle E_6\rangle_\varphi=R_6$. On the other hand,
\begin{equation}\label{eq:chain}
 \Delta(G\omega_\varphi(G))=Gw^2(D^2\omega_\varphi+D\omega_\varphi).
\end{equation}
Multiply \eqref{eq:bochner} by $G\omega_\varphi$ and integrate twice by parts to get
\[
 \langle E_1\rangle_\varphi-\frac{2(n-2)}{n-1}\langle E_6\rangle_\varphi
 =\tfrac12\int(D^2\omega_\varphi+D\omega_\varphi)\dmu.
\]
Substituting $R_6$ gives $R_1$.

\emph{The quadratic stress tensor.} Set $T=a^2g/2-B^2$. By \eqref{eq:divergences},
\[
 0=\int\divg\bigl(\omega_\varphi T(\nabla G,\cdot)^\sharp\bigr)\dV.
\]
Harmonicity of $G$ and \eqref{eq:variables} give
\[
 T:\nabla^2G=-Gw^4\left(S^2:Z+\frac n{n-1}U-\frac1{n-1}b^2\right)=-Gw^4E_2.
\]
The remaining product-rule term is
\[
 \omega_\varphi'(G)T(\nabla G,\nabla G)=Gw^4(\tfrac12b^2-U)D\omega_\varphi,
\]
proving $\langle E_2\rangle_\varphi=R_2$.

\emph{Simons' identity.} Equation~\eqref{eq:simons} is $\tfrac12\Delta a^2=-w^4E_3$. Multiply by $G\omega_\varphi$ and use \eqref{eq:chain}, giving $\langle E_3\rangle_\varphi=R_3$.

\emph{The curvature stability test.} We use the Simons--stability mechanism of \cite{SSY75}, keeping all CMC and ambient-curvature terms. For compactly supported Lipschitz $f$, expand $Q_c(af)$ and integrate the gradient of $a^2$ by parts:
\[
 Q_c(af)=\int a^2|\nabla f|^2\dV
 +\int f^2\left(|\nabla a|^2-\tfrac12\Delta a^2-(a^2+n(c+H^2))a^2\right)\dV.
\]
Using \eqref{eq:simons}, this becomes
\begin{equation}\label{eq:curvtest}
 Q_c(af)=\int a^2|\nabla f|^2\dV
 -\int f^2\bigl(|\nabla B|^2-|\nabla a|^2
                 +2n(c+H^2)a^2+nH\tr B^3\bigr)\dV.
\end{equation}
In particular, the term is $2n(c+H^2)a^2$, retaining both the Simons and stability contributions. Set $f=\sqrt G\,\varphi^2$. Then
\begin{equation}\label{eq:gradtest}
 |\nabla f|^2=Gw^2\left(\tfrac14\omega_\varphi
                 +2\varphi^3D\varphi+4\varphi^2(D\varphi)^2\right).
\end{equation}
The inequality $Q_c(af)\ge0$ is exactly $\langle E_4\rangle_\varphi\le R_4$.

\emph{The Green-gradient stability test.} Equation~\eqref{eq:gradw} implies
\[
 \nabla(w\sqrt G)=\sqrt G\,w^2(Z\nu+\nu/2).
\]
Expanding $Q_c(w\sqrt G\,\varphi^2)\ge0$ gives
\[
 \langle E_5\rangle_\varphi
 \le\int\bigl(\omega_\varphi E_6+(D\omega_\varphi)E_6
                +2\varphi^3D\varphi+4\varphi^2(D\varphi)^2\bigr)\dmu.
\]
Apply \eqref{eq:graddiv} with $\eta=\omega_\varphi$ and then $\eta=D\omega_\varphi$. Since
\begin{equation}\label{eq:Domega}
 D\omega_\varphi=4\varphi^3D\varphi,\qquad
 D^2\omega_\varphi=12\varphi^2(D\varphi)^2+4\varphi^3D^2\varphi,
\end{equation}
the right-hand side reduces to $R_5$.

\emph{The linear CMC divergence.} Because $\divg(HB)=0$ and $\tr B=0$,
\[
 HB:\nabla^2G=Gw^4\left(hS:Z+\frac n{n-1}hV\right)=Gw^4E_7.
\]
Integration of $\divg(\omega_\varphi HB(\nabla G,\cdot)^\sharp)$ yields
\[
 0=\langle E_7\rangle_\varphi+\int hV D\omega_\varphi\dmu,
\]
which gives $R_7$.

\emph{The constant-mean-curvature stability test.} Finally apply stability to $H\sqrt G\,\varphi^2$. Multiplication of \eqref{eq:gradtest} by $H^2$ gives $\langle E_8\rangle_\varphi\le R_8$.

Every integration above concerns smooth unnormalized tensors or a compactly supported $W^{1,2}$ stability test. Approximation on a fixed slab justifies the norm calculations at $a=0$ and $w=0$, as explained after \eqref{eq:unnormalized}. No limit at either end of the Green range has been taken.
\end{proof}

\section{Dimension-dependent coercivity}\label{sec:algebra}
Throughout this section $2\le n\le6$. All tensor estimates are pointwise in an
$n$-dimensional inner-product space. 

\subsection{Spectral and Codazzi estimates}
\begin{lemma}[Spectral estimates]\label{lem:spectral}
Let $S$ be symmetric and trace-free, set $b=|S|$ and $D_0=S^2-b^2g/n$, and let
$\nu$ be a unit vector. Then
\begin{align}
 \|S\|_{\op}^2&\le\frac{n-1}{n}b^2,
 &\tr S^4&\le\frac{n^2-3n+3}{n(n-1)}b^4,\label{eq:spectral1}\\
 |D_0|^2&\le d_n b^4,
 &\|D_0\|_{\op}&\le e_n b^2,\label{eq:spectral2}
\end{align}
where
\[
 d_n=\frac{(n-2)^2}{n(n-1)},\qquad e_n=\frac{n-2}{n}.
\]
Consequently
\begin{equation}\label{eq:spectral3}
 |\tr S^3|\le\sqrt{d_n}\,b^3,\qquad
 |\langle D_0\nu,S\nu\rangle|
 \le e_n\sqrt{\frac{n-1}{n}}\,b^3.
\end{equation}
For $n=2$ one has the exact identities $S^2=b^2g/2$, $D_0=0$ and $\tr S^3=0$.
\end{lemma}
\begin{proof}
For eigenvalues $\lambda_i$ with sum zero,
$\lambda_i^2=(\sum_{j\ne i}\lambda_j)^2\le(n-1)\sum_{j\ne i}\lambda_j^2$,
which proves the operator bound. If $b>0$, normalize $\sum\lambda_i^2=1$ and
maximize $q=\sum\lambda_i^4$. The two constraint gradients are independent.
Lagrange multipliers give $4\lambda_i^3-2a\lambda_i-\beta=0$; multiplication
by $\lambda_i$ and summation gives $a=2q$. There are at most three distinct
values. If there are three, they are the roots $r_1,r_2,r_3$ of
$x^3-qx-\beta/4$. Thus $q=(r_1^2+r_2^2+r_3^2)/2\le1/2$, since each root
occurs at least once in the normalized spectrum. If there are two values
with multiplicities $k,n-k$, solving the two constraints gives
\[
 q=\frac{n}{k(n-k)}-\frac3n
 \le\frac{n}{n-1}-\frac3n
 =\frac{n^2-3n+3}{n(n-1)}.
\]
For $n\ge3$ the latter bound is at least $1/2$; a single eigenvalue is
impossible. The case $n=2$ follows directly from the spectrum $(t,-t)$.
The case $b=0$ is immediate.

Now $|D_0|^2=\tr S^4-b^4/n\le d_n b^4$. For $n\ge3$, every eigenvalue
of $D_0$ lies between $-b^2/n$ and $(n-2)b^2/n$, proving its operator bound.
For $n=2$, use $D_0=0$, rather than this interval bound. Finally
$\tr S^3=D_0:S$, so Cauchy--Schwarz gives the first inequality in
\eqref{eq:spectral3}; the two operator estimates give the second.
\end{proof}

\begin{lemma}[Refined Codazzi--Kato inequality]\label{lem:kato}
If $B$ is a trace-free Codazzi tensor on an $n$-manifold, then
\begin{equation}\label{eq:kato}
 |\nabla B|^2\ge\frac{n+2}{n}|\nabla|B||^2
 \quad\text{almost everywhere}.
\end{equation}
\end{lemma}
\begin{proof}
At $a=|B|>0$ set $\widehat B=B/a$. On fully symmetric trace-free
three-tensors define $(L_{\widehat B}T)_i=\sum_{j,k}\widehat B_{jk}T_{ijk}$.
Its adjoint on that tensor space is
\begin{equation}\label{eq:adjoint}
\begin{aligned}
 (L_{\widehat B}^*v)_{ijk}
 &=\tfrac13(v_i\widehat B_{jk}+v_j\widehat B_{ik}+v_k\widehat B_{ij})\\
 &\quad-\frac{2}{3(n+2)}
 (\delta_{ij}(\widehat Bv)_k+\delta_{ik}(\widehat Bv)_j
                      +\delta_{jk}(\widehat Bv)_i).
\end{aligned}
\end{equation}
The trace of the first line is $2\widehat Bv/3$, while the trace of the
parenthesized tensor is $(n+2)\widehat Bv$. If $T_0$ denotes the first line
and $K_0$ the parenthesized tensor, direct contractions give
\[
 |T_0|^2=\tfrac13|v|^2+\tfrac23|\widehat Bv|^2,\quad
 |K_0|^2=3(n+2)|\widehat Bv|^2,\quad
 \langle T_0,K_0\rangle=2|\widehat Bv|^2.
\]
Therefore
\[
 |L_{\widehat B}^*v|^2
 =\tfrac13|v|^2+\frac{2n}{3(n+2)}|\widehat Bv|^2
 \le\frac n{n+2}|v|^2.
\]
Since $\nabla B$ is fully symmetric and trace-free and
$L_{\widehat B}(\nabla B)=\nabla a$, this proves \eqref{eq:kato} where
$a>0$. On $\{a=0\}$ its weak gradient vanishes almost everywhere, completing
the proof. This last fact is the level-set property of Sobolev weak gradients;
see \cite[Chapter~4]{EG15}.
\end{proof}

\subsection{A weighted Hessian minimum}
\begin{lemma}[Exact weighted Hessian minimum]\label{lem:variableHessian}
Let $n\ge2$, $0<\upsilon<\min\{2,n/(n-1)\}$, and set
\[
 \mathfrak D=n-(n-1)\upsilon,\qquad
 p_\upsilon=\frac{\upsilon}{2(2-\upsilon)},\qquad
 q_\upsilon=\frac{\upsilon^2(n-2)}{4\mathfrak D(2-\upsilon)}.
\]
For trace-free symmetric endomorphisms $L,Z$ and a unit vector $\nu$,
\begin{equation}\label{eq:variableHessian}
 |Z|^2-\upsilon|Z\nu|^2+L:Z
 \ge-\frac14|L|^2-p_\upsilon|L\nu|^2-q_\upsilon L(\nu,\nu)^2.
\end{equation}
The right-hand side is the exact minimum over $Z$ for fixed $L,\nu$.
\end{lemma}
\begin{proof}
Choose $\nu=e_1$ and write
\[
 Z=\begin{pmatrix}z&v^T\\v&W-zI_{n-1}/(n-1)\end{pmatrix},\qquad
 L=\begin{pmatrix}\ell_0&\eta_0^T\\\eta_0&V_0-\ell_0I_{n-1}/(n-1)\end{pmatrix},
 \qquad \tr W=\tr V_0=0.
\]
The expression to minimize is
\[
 \frac{\mathfrak D}{n-1}z^2+(2-\upsilon)|v|^2+|W|^2
 +\frac n{n-1}\ell_0z+2\eta_0\cdot v+V_0:W.
\]
All three quadratic coefficients are positive. Completing the three squares gives the minimum
\[
 -\frac{n^2\ell_0^2}{4(n-1)\mathfrak D}
 -\frac{|\eta_0|^2}{2-\upsilon}-\frac14|V_0|^2.
\]
Substitute
\[
 |L|^2=\frac n{n-1}\ell_0^2+2|\eta_0|^2+|V_0|^2,
 \qquad |L\nu|^2=\ell_0^2+|\eta_0|^2
\]
to obtain \eqref{eq:variableHessian}. The calculation also gives
\begin{equation}\label{eq:thetaVariable}
 p_\upsilon+q_\upsilon=\frac{n\upsilon}{4\mathfrak D}.
\end{equation}
When $n=2$, the transverse trace-free blocks $W,V_0$ vanish; the same calculation applies.
\end{proof}

\subsection{One coefficient choice for dimensions three through six}\label{subsec:coefficients}
For $3\le n\le6$, choose the coefficients in Table~\ref{tab:coefficients}. The thresholds are those of Theorem~\ref{thm:hyperbolic}.

\begin{table}[htbp]
\centering
\caption{Thresholds and the free coefficients of the common combination.}\label{tab:coefficients}
\begin{tabular}{ccccc}
\toprule
$n$&$\tau_n$&$\alpha$&$\chi$&$\ell$\\\midrule
3&$251/250$&$19/37$&$4/25$&$4/15$\\
4&$207/200$&$25/42$&$93/250$&$127/50$\\
5&$29/25$&$20/31$&$7/12$&$51/5$\\
6&$5/3$&$55/81$&$31/40$&$57/2$\\\bottomrule
\end{tabular}
\end{table}

The remaining coefficients are
\begin{align}
 \beta&=\frac{2\chi}{n+2},&
 \epsilon&=-(n-2)\alpha,&
 \upsilon&=\frac1n+\frac\chi4,\notag\\
 \mathfrak D&=n-(n-1)\upsilon,&
 \zeta&=\frac{2\mathfrak D}{n-1}-\frac{2\mathfrak D}{n\alpha},&
 \theta&=\frac{n\upsilon}{4\mathfrak D}.
 \label{eq:coefficients}
\end{align}
All these quantities are fixed by $n$; in particular $\zeta$ and $\epsilon$ are coefficients, not cutoffs or limiting parameters. Put
\begin{equation}\label{eq:P}
 P_n=E_1+\alpha E_2+\beta E_3+\chi E_4
          +\upsilon E_5+\zeta E_6+\epsilon E_7+\ell E_8.
\end{equation}
Each row has $\chi,\upsilon,\ell>0$ and
$0<\upsilon<\min\{2,n/(n-1)\}$. Hence all three one-sided relations in Proposition~\ref{prop:identities} have positive coefficients. Set
\begin{equation}\label{eq:Gamma}
 k_n=n-1+n\upsilon,\qquad j_n=n(2\chi-\beta),\qquad
 \Gamma_n(b,h)=k_n+Jb^2+n\ell h^2.
\end{equation}
Thus $k_n,j_n,\ell$ and $\Gamma_n$ are positive.

\begin{proposition}[A common pointwise lower bound]\label{prop:coercive}
For $3\le n\le6$, let $P_{n,0}$ be the algebraic specialization of \eqref{eq:P} at $\kappa=0$, with all other variables fixed. Set
\begin{equation}\label{eq:worstcase}
 \delta_n=\tau_n^{-1},\qquad
 \widehat P_n=P_{n,0}-\delta_nh^2\Gamma_n(b,h).
\end{equation}
For $h\ge0$ and the trace-free and Codazzi data of Section~\ref{sec:identities},
\begin{equation}\label{eq:worstbound}
 \widehat P_n\ge\frac1{10000}(1+b^4+h^4).
\end{equation}
Consequently,
\begin{equation}\label{eq:coercive}
 \begin{cases}
 P_n\ge\dfrac{1+b^4+h^4}{10000}+\kappa\Gamma_n(b,h),&c\in\{0,1\},\\[2mm]
 P_n\ge\dfrac{1+b^4+h^4}{10000},&c=-1,\quad H^2\ge\tau_n.
 \end{cases}
\end{equation}
\end{proposition}
\begin{proof}
The ambient-curvature terms in \eqref{eq:E1}--\eqref{eq:E8} give the exact identity
\begin{equation}\label{eq:curvatureCorrection}
 P_n-P_{n,0}=\kappa\Gamma_n(b,h).
\end{equation}
We first prove \eqref{eq:worstbound}, a pointwise algebraic estimate independent of the ambient curvature. Put
\[
 D_0=S^2-b^2g/n,\qquad
 L_0=\alpha D_0+\epsilon hS,\qquad
 \Pi=\nu\otimes\nu-g/n.
\]
Before minimizing over $Z$, the derivative term is
\begin{equation}\label{eq:KatoMatch}
 (\chi-\beta)|\mathcal C|^2-\chi|Y|^2
 =\chi\left(\frac n{n+2}|\mathcal C|^2-|Y|^2\right)\ge0
\end{equation}
by Lemma~\ref{lem:kato}. The Hessian terms are
$|Z|^2-\upsilon|Z\nu|^2+(L_0+\zeta\Pi):Z$.
In Lemma~\ref{lem:variableHessian}, the shift by $\zeta\Pi$ contributes
\begin{equation}\label{eq:shift}
 -\frac{\zeta^2(n-1)}{4\mathfrak D}
 -\frac{n\zeta}{2\mathfrak D}L_0(\nu,\nu)
\end{equation}
besides the negative quadratic terms in $L_0$. To verify the shift, use
$|\Pi|^2=(n-1)/n$, $\Pi\nu=(n-1)\nu/n$, $\tr L_0=0$, and \eqref{eq:thetaVariable}. Since $L_0(\nu,\nu)^2\le|L_0\nu|^2$, the remaining quadratic terms are bounded below by
\[
 -\frac14|L_0|^2-\theta|L_0\nu|^2.
\]
The coefficients of $U,b^2,hV$ cancel exactly:
\begin{align}
 -1+\frac{n\alpha}{n-1}-\frac{n\zeta\alpha}{2\mathfrak D}&=0,\label{eq:cancelU}\\
 \upsilon-\frac\chi4-\frac{\alpha}{n-1}
                       +\frac{\zeta\alpha}{2\mathfrak D}&=0,\label{eq:cancelb}\\
 (n-2)+\frac{n\epsilon}{n-1}
                       -\frac{n\zeta\epsilon}{2\mathfrak D}&=0.\label{eq:cancelhV}
\end{align}
These identities follow directly from \eqref{eq:coefficients}. In particular, none of these quadratic terms is estimated separately by Young's inequality.

Using $d_n,e_n$ from Lemma~\ref{lem:spectral}, define six scalar coefficients:
\begin{align}
 c_0&=\frac1{n-1}-\frac\upsilon4
            -\frac{\zeta^2(n-1)}{4\mathfrak D},&
 c_2&=\frac\ell4-(1-\delta_n)k_n,\notag\\
 c_4&=\beta-\alpha^2\left(\frac{d_n}4+\theta e_n^2\right),&
 c_{04}&=n\ell(1-\delta_n),\label{eq:polynomialCoefficients}\\
 c_{31}&=\left[n(\chi-\beta)-\frac{\alpha\epsilon}{2}\right]\sqrt{d_n}
       -2\theta\alpha\epsilon e_n\sqrt{\frac{n-1}{n}},\notag\\
 c_{22}&=\ell+(1-\delta_n)j_n
       -\epsilon^2\left(\frac14+\theta\frac{n-1}{n}\right).\notag
\end{align}
These are coefficients of a scalar polynomial; their subscripts distinguish them from the ambient curvature $c$ and from the tensors $A,B,\mathcal C$. For instance, $c_{31}$ multiplies $b^3h$ and $c_{22}$ multiplies $b^2h^2$.

To derive them, expand the two remaining negative norms:
\begin{align*}
 -\tfrac14|L_0|^2-\theta|L_0\nu|^2
 ={}&-\tfrac{\alpha^2}{4}|D_0|^2-\theta\alpha^2|D_0\nu|^2
       -\tfrac{\alpha\epsilon}{2}h\tr S^3
       -2\theta\alpha\epsilon h\langle D_0\nu,S\nu\rangle\\
 &-\tfrac{\epsilon^2}{4}h^2b^2-\theta\epsilon^2h^2U.
\end{align*}
The resulting cubic coefficients satisfy
$n(\chi-\beta)-\alpha\epsilon/2>0$ and $-2\theta\alpha\epsilon\ge0$.
Apply
\[
 |D_0|^2\le d_nb^4,\quad |D_0\nu|^2\le e_n^2b^4,\quad
 |\tr S^3|\le\sqrt{d_n}\,b^3,\quad U\le\frac{n-1}{n}b^2,
\]
and
\[
 |\langle D_0\nu,S\nu\rangle|
 \le e_n\sqrt{\frac{n-1}{n}}\,b^3.
\]
Using $h\ge0$, \eqref{eq:cancelU}--\eqref{eq:cancelhV}, and the subtraction in \eqref{eq:worstcase}, we obtain
\begin{equation}\label{eq:polynomial}
 \widehat P_n\ge
 c_0-c_2h^2+c_4b^4-c_{31}hb^3+c_{22}h^2b^2+c_{04}h^4.
\end{equation}

The positive scalar margins can be stated without implicit sign choices. Define
\begin{equation}\label{eq:discriminants}
 \Delta_n^{(h)}=4c_0c_{04}-c_2^2,\qquad
 \Delta_n^{(bh)}=4c_4c_{22}-c_{31}^2,
\end{equation}
and
\begin{equation}\label{eq:margins}
 \xi_n=\frac{\Delta_n^{(h)}}{8c_0},\qquad
 r_n=\frac{c_0\Delta_n^{(h)}}{4c_0c_{04}+c_2^2},\qquad
 s_n=\frac{\Delta_n^{(bh)}}{4c_{22}}.
\end{equation}
Lemma~\ref{lem:arithmetic} verifies $c_0,c_{22},c_{04}>0$ and the following strict rational comparisons by exact substitution:
\begin{equation}\label{eq:marginBounds}
\begin{array}{c|ccc}
 n&r_n>&s_n>&\xi_n>\\\hline
 3&43/1000&16/100000&39/100000\\
 4&4/1000&48/100000&75/10000\\
 5&7/10000&13/100000&44/1000\\
 6&1/10000&14/100000&1/10
\end{array}
\end{equation}
In particular $c_{04}-\xi_n=(c_{04}+c_2^2/(4c_0))/2>0$.
Completing the squares gives the identities
\begin{align}
 c_4b^4-c_{31}hb^3+c_{22}h^2b^2
 &=s_nb^4+c_{22}b^2\left(h-\frac{c_{31}}{2c_{22}}b\right)^2,\notag\\
 c_0-c_2h^2+c_{04}h^4
 &=r_n+\xi_nh^4+(c_{04}-\xi_n)
      \left(h^2-\frac{c_2}{2(c_{04}-\xi_n)}\right)^2.
 \label{eq:squares}
\end{align}
Together with \eqref{eq:marginBounds}, these prove \eqref{eq:worstbound}.

Finally, \eqref{eq:curvatureCorrection} and \eqref{eq:worstcase} give
\begin{equation}\label{eq:threeSpaceTransfer}
 P_n=\widehat P_n+(\kappa+\delta_nh^2)\Gamma_n(b,h).
\end{equation}
For $c\ge0$, the second term retains at least $\kappa\Gamma_n$.
For $c=-1$ with $H^2\ge\tau_n$, write $\delta=H^{-2}\le\delta_n$. Then
$\kappa+\delta_nh^2=(\delta_n-\delta)h^2\ge0$.
This proves both cases of \eqref{eq:coercive}, including the hyperbolic endpoint. The case $H=0$ for $c\ge0$ is obtained by setting $h=0$ in the same algebraic identities; no division by $H$ is used in that case.
\end{proof}

\subsection{The surface combinations}
\begin{proposition}[Coercivity in dimension two]\label{prop:surfaces}
Assume $n=2$ and $h\ge0$. For $c\in\{0,1\}$, set
\begin{equation}\label{eq:surfacePositive}
 P_2=E_1+\frac12E_5+4E_8.
\end{equation}
Then
\begin{equation}\label{eq:surfacePositiveBound}
 P_2\ge\frac78(1+h^4)+2\kappa.
\end{equation}
For $c=-1$ and $H^2>1$, set $\eta=1-H^{-2}>0$ and
\begin{equation}\label{eq:surfaceNegative}
 P_2=E_1+\frac12E_5+4\eta E_8.
\end{equation}
Then
\begin{equation}\label{eq:surfaceNegativeBound}
 P_2\ge\frac78+8\eta^2h^4
 \ge\varepsilon_{2,H}(1+h^4),\qquad
 \varepsilon_{2,H}=\min\{7/8,8\eta^2\}>0.
\end{equation}
\end{proposition}
\begin{proof}
For a trace-free symmetric two-tensor in dimension two,
$S^2=b^2g/2$ and $U=b^2/2$. Direct substitution gives, for $c\ge0$,
\begin{align*}
 P_2={}&|Z|^2-\tfrac12|Z\nu|^2+\tfrac78+h^2+2\kappa
              +4h^2b^2+8h^4+8\kappa h^2.
\end{align*}
Since $|Z\nu|\le|Z|$ and $\kappa\ge0$, discarding the other nonnegative terms and using $8\ge7/8$ proves \eqref{eq:surfacePositiveBound}.
For $c=-1$, use $h^2+\kappa=\eta h^2$ to obtain
\[
 P_2=|Z|^2-\tfrac12|Z\nu|^2+\tfrac78
              +\eta h^2+4\eta h^2b^2+8\eta^2h^4.
\]
This proves \eqref{eq:surfaceNegativeBound}. These combinations control $1+h^4$, not $b^4$. Their cutoff errors involve only $R_1,R_5,R_8$, so that weaker control is sufficient in Section~\ref{sec:cutoff}.
\end{proof}

\section{Logarithmic cutoffs and proofs of the main results}\label{sec:cutoff}
We continue under the contradiction assumptions introduced at the beginning of Section~\ref{sec:green}. Define
\begin{equation}\label{eq:cutoffDensity}
 \mathcal L_n=\begin{cases}1+h^4,&n=2,\\1+b^4+h^4,&3\le n\le6,\end{cases}
 \qquad
 \varepsilon_*=
 \begin{cases}
 7/8,&n=2,\ c\ge0,\\
 \min\{7/8,8(1-H^{-2})^2\},&n=2,\ c=-1,\\
 1/10000,&3\le n\le6.
 \end{cases}
\end{equation}
The number $\varepsilon_*$ is positive under the relevant hypotheses and independent of the cutoff. The retained ambient density is
\begin{equation}\label{eq:ambientDensity}
 \mathcal K_n=\begin{cases}
 2\kappa,&n=2,\ c\ge0,\\
 \kappa\Gamma_n(b,h),&3\le n\le6,\ c\ge0,\\
 0,&c=-1.
 \end{cases}
\end{equation}
Thus $\mathcal K_n\ge0$. Its definition in the hyperbolic case is used only after the negative ambient terms have been absorbed in Proposition~\ref{prop:coercive} or~\ref{prop:surfaces}.

\begin{proposition}[A common compact-test estimate]\label{prop:cutoff}
There is a constant $C_{n,H}<\infty$, independent of the cutoff, such that every nonnegative $\varphi\in C_c^\infty((0,t_0))$ satisfies
\begin{equation}\label{eq:cutoff}
 \int_{\mathcal O_G}
 \left(\frac{\varepsilon_*}{2}\mathcal L_n+\mathcal K_n\right)
 \varphi(G)^4\dmu
 \le C_{n,H}\int_{\mathcal O_G}
 \bigl(|D\varphi|^4+|D^2\varphi|^4\bigr)\dmu.
\end{equation}
\end{proposition}
\begin{proof}
First suppose $3\le n\le6$. All one-sided relations in Proposition~\ref{prop:identities} enter \eqref{eq:P} with positive coefficients, so
\begin{equation}\label{eq:combinedErrors}
 \int P_n\varphi^4\dmu
 \le R_1+\alpha R_2+\beta R_3+\chi R_4
                    +\upsilon R_5+\zeta R_6+\epsilon R_7+\ell R_8.
\end{equation}
All integrals here are over $\mathcal O_G$ and all scalar cutoff factors are composed with $G$.
Using \eqref{eq:Domega}, $U\le b^2$, and
$|hV|\le(b^2+h^2)/2$, the absolute values of the errors give
\begin{equation}\label{eq:errorbound}
 \int P_n\varphi^4\dmu
 \le C_n\int(1+b^2+h^2)
 \left[\varphi^3\bigl(|D\varphi|+|D^2\varphi|\bigr)
                  +\varphi^2|D\varphi|^2\right]\dmu.
\end{equation}
No additional ambient-curvature error occurs.

Set $L=1+b^4+h^4$. Then $1+b^2+h^2\le\sqrt3\,L^{1/2}$ and $L\ge1$.
For $d\ge0$ and $\varepsilon>0$, Young's inequality gives
\begin{align*}
 L^{1/2}\varphi^3d
 &\le\varepsilon L\varphi^4+C_\varepsilon\varphi^2d^2
 \le2\varepsilon L\varphi^4+C'_\varepsilon d^4,\\
 L^{1/2}\varphi^2d^2
 &\le\varepsilon L\varphi^4+C_\varepsilon d^4.
\end{align*}
Apply these inequalities with $d=|D\varphi|$ and $d=|D^2\varphi|$.
Choose $\varepsilon$ sufficiently small that the absorbed term is at most
$\varepsilon_*L\varphi^4/2$. Proposition~\ref{prop:coercive} then leaves
$\varepsilon_*L\varphi^4/2$ and, when $c\ge0$, all of
$\kappa\Gamma_n\varphi^4$ on the left. This proves \eqref{eq:cutoff} for $n\ge3$.

For $n=2$, use Proposition~\ref{prop:surfaces}. The error combination is
$R_1+R_5/2+4R_8$ when $c\ge0$, and $R_1+R_5/2+4\eta R_8$ when $c=-1$.
It satisfies \eqref{eq:errorbound} with $1+b^2+h^2$ replaced by
$1+h^2\le\sqrt2(1+h^4)^{1/2}$. The same absorption proves \eqref{eq:cutoff}.
In particular no $b^2$ error remains uncontrolled. For hyperbolic surfaces the constant may depend on the fixed value of $H$; uniformity as $H^2\downarrow1$ is not needed.
\end{proof}

The left-hand densities have explicit unnormalized meanings. In every dimension,
\[
 \mathcal L_n\dmu
 =\left[\frac{|\nabla G|^4}{G^3}+H^4G
             +\mathbf1_{\{n\ge3\}}|B|^4G\right]\dV.
\]
For $c\ge0$ and $n\ge3$, the other density is
\[
 \mathcal K_n\dmu
 =c\left[k_n\frac{|\nabla G|^2}{G}+j_n|B|^2G+n\ell H^2G\right]\dV;
\]
for $n=2$ it is $2c|\nabla G|^2G^{-1}\dV$.
These formulas also specify the estimates across the critical set of $G$.

\begin{lemma}[Pushforward of the unit flux]\label{lem:flux}
On $\mathcal O_G$, put $s=\log(t_0/G)$ and $t(s)=t_0e^{-s}$. For every nonnegative measurable $\beta$ on $(0,\infty)$,
\begin{equation}\label{eq:fluxdensity}
 \int_{\mathcal O_G}\beta(s)\frac{|\nabla G|^2}{G}\dV
 =\int_0^\infty\beta(s)\dd s.
\end{equation}
Consequently, as locally finite measures on $(0,\infty)$ (signed for $\kappa\mu$ when $c=-1$),
\begin{equation}\label{eq:fluxes}
 s_\#(h^2\mu)=H^2\,\mathrm ds,
 \qquad s_\#(\kappa\mu)=c\,\mathrm ds,
 \qquad s_\#\mu\le C_G^2\,\mathrm ds.
\end{equation}
\end{lemma}
\begin{proof}
Let $q=|\nabla G|$. Coarea and \eqref{eq:unitflux} give
\begin{align*}
 \int_{\mathcal O_G}\beta(s)\frac{q^2}{G}\dV
 &=\int_0^{t_0}\frac{\beta(\log(t_0/t))}{t}
                \left(\int_{\{G=t\}}q\dAg\right)\dd t\\
 &=\int_0^{t_0}\beta(\log(t_0/t))\frac{\dd t}{t}
  =\int_0^\infty\beta(s)\dd s.
\end{align*}
Tonelli's theorem permits infinite values. Critical levels do not affect the formula, since regular values have full measure and the coarea density vanishes at $q=0$.

Now
\[
 h^2\dmu=H^2\frac{q^2}{G}\dV,\qquad
 \kappa\dmu=c\frac{q^2}{G}\dV,\qquad
 \dmu=w^2\frac{q^2}{G}\dV.
\]
The first two measure identities follow from \eqref{eq:fluxdensity}, while $w\le C_G$ gives the third. Here $s_\#$ denotes measure pushforward: for a measure $\nu_*$ on $\mathcal O_G$, $s_\#\nu_*$ is characterized by $\int\beta\dd(s_\#\nu_*)=\int\beta(s(x))\dd\nu_*(x)$.
\end{proof}

\begin{proof}[Proof of Theorem~\ref{thm:bounded}]
We argue under \eqref{eq:positiveGamma}. Choose a fixed $\psi\in C_c^\infty((1,4))$ with $0\le\psi\le1$ and $\psi=1$ on $[2,3]$. For $R\ge1$, define
\begin{equation}\label{eq:longcutoff}
 \varphi_R(t)=\psi\left(\frac{\log(t_0/t)}R\right).
\end{equation}
For each finite $R$, its support lies in a compact positive Green interval. Since $D=-\partial_s$,
\[
 |D\varphi_R|\le C_\psi R^{-1},\qquad
 |D^2\varphi_R|\le C_\psi R^{-2},
\]
and these derivatives are supported where $R<s<4R$.
Lemma~\ref{lem:flux} therefore bounds the right-hand side of \eqref{eq:cutoff} by
\begin{equation}\label{eq:uppercutoff}
 C_{n,H}C_G^2(R^{-3}+R^{-7}).
\end{equation}
On the left, $\mathcal L_n\ge1+h^4\ge2h^2$.
If $c\ge0$, use $k_n$ from \eqref{eq:Gamma} for $n\ge3$ and set $k_2=2$.
Then $\mathcal K_n\ge k_n\kappa$, and \eqref{eq:fluxes} gives
\begin{equation}\label{eq:lowercutoff}
 \int\left(\frac{\varepsilon_*}{2}\mathcal L_n+\mathcal K_n\right)
          \varphi_R^4\dmu
 \ge(\varepsilon_*H^2+k_nc)
       \int_0^\infty\psi(s/R)^4\dd s
 \ge(\varepsilon_*H^2+k_nc)R.
\end{equation}
This coefficient is strictly positive whenever $c+H^2>0$.
If $c=-1$, the hypotheses \eqref{eq:threshold} ensure $H^2>1$ and $\varepsilon_*>0$, and the lower bound is instead
\[
 \frac{\varepsilon_*}{2}\int\mathcal L_n\varphi_R^4\dmu
 \ge\varepsilon_*H^2\int_0^\infty\psi(s/R)^4\dd s
 \ge\varepsilon_*H^2R.
\]
In each case the linear lower bound contradicts \eqref{eq:uppercutoff} as $R\to\infty$. Hence the noncompact bounded-curvature alternative in Theorem~\ref{thm:bounded} has the stated restrictions.
\end{proof}

\begin{remark}
For $c=0$, the contradiction is driven by $H^2\,\mathrm ds$; for $c=1$, the ambient-curvature flux remains when $H=0$. For $c=-1$, negative ambient terms are absorbed in the pointwise estimate and the positive $H^2\,\mathrm ds$ flux gives the contradiction. Both estimates use the same compactly supported function. No global integrability of $G|B|^4$ or $G$ is assumed, and no boundary contribution at infinity is discarded.
\end{remark}

\subsection{Proofs of the main results and stability corollaries}
\begin{proof}[Proof of Theorem~\ref{thm:sphere}]
Finite strong index gives exterior stability by Lemma~\ref{lem:exterior} and bounded curvature by Lemma~\ref{lem:curvature}. If $M$ were noncompact, Theorem~\ref{thm:bounded} would imply $1+H^2=0$, which is impossible. Thus $M$ is compact.
\end{proof}
\begin{proof}[Proof of Theorem~\ref{thm:euclidean}]
Apply the same two lemmas and then Theorem~\ref{thm:bounded} with $c=0$. Since $M$ is noncompact, the conclusion is $H^2=0$.
\end{proof}
\begin{proof}[Proof of Theorem~\ref{thm:hyperbolic}]
Finite strong index supplies exterior stability and, by Lemma~\ref{lem:curvature}, bounded second fundamental form. If $M$ were noncompact, Theorem~\ref{thm:bounded} would contradict the relevant condition in \eqref{eq:threshold}. Thus $M$ is compact.
\end{proof}
\begin{proof}[Proof of Corollary~\ref{cor:constrained}]
The index comparison \eqref{eq:indexcomparison} turns finite volume-constrained index into finite strong index. Apply the appropriate main theorem.
\end{proof}
\begin{proof}[Proof of Corollary~\ref{cor:strong}]
In the sphere, strong stability gives index zero and hence compactness by Theorem~\ref{thm:sphere}. Only at this point use the constant test function $1$:
\[
 Q_1(1)=-\int_M(|A|^2+n)\dV<0,
\]
a contradiction.

In $\mathbb H^{n+1}(-1)$ under \eqref{eq:threshold}, strong stability gives index zero, so Theorem~\ref{thm:hyperbolic} makes $M$ compact. Now
\[
 Q_{-1}(1)=-\int_M\bigl(|B|^2+n(H^2-1)\bigr)\dV<0,
\]
again a contradiction.

In Euclidean space, a compact strongly stable immersion would satisfy $Q_0(1)=-\int_M|A|^2\dV\ge0$, so $A\equiv0$. Then $H=0$ and $\Delta|F|^2=2n$, impossible on a compact manifold without boundary. Thus $M$ is noncompact. Theorem~\ref{thm:euclidean} gives $H=0$, and Theorem~\ref{thm:HLW} gives an affine hyperplane. The immersion is a local isometry onto that plane. Completeness makes it a covering map, and the plane is simply connected, so the covering is trivial.
\end{proof}
\begin{proof}[Proof of Corollary~\ref{cor:weak}]
Weak stability means $\ind_0(Q_1)=0$, so \eqref{eq:indexcomparison} gives $\ind(Q_1)\le1$. Theorem~\ref{thm:sphere} makes $M$ compact. The compact weakly stable CMC classification \cite[Theorem~1.2]{BCE88} applies to this two-sided immersion and gives a geodesic $n$-sphere as its image. The immersion onto that image is a local isometry from a complete connected manifold, hence a covering. Since an $n$-sphere with $n\ge2$ is simply connected, the covering is trivial.
\end{proof}
\begin{remark}[Finite index versus stability]
The general finite-index conclusions are compactness in the sphere and minimality in the noncompact Euclidean case, not a classification of all finite-index immersions. For example, non-equatorial geodesic spheres in $\mathbb S^{n+1}(1)$ have nonzero constant mean curvature and finite strong index. Strong and weak stability are different: geodesic spheres are weakly stable but fail the constant-function strong-stability test.
\end{remark}

\section{Examples and dimensional scope}\label{sec:examples}
\subsection{Stable hyperbolic tubes}\label{sec:tubes}
In this section, $n$ denotes the intrinsic dimension, so that the ambient dimension is $n+1$. The main theorems above cover $2\le n\le6$. We compute the geometry and stability of a classical tube family explicitly, to identify endpoint phenomena and the need for dimension-dependent thresholds.

\begin{remark}[The surface endpoint]\label{rem:surfaceEndpoint}
The strict inequality $H^2>1$ in dimension two cannot be replaced by
$H^2\ge1$. A horosphere in $\mathbb H^3(-1)$ has induced Euclidean metric,
$A=g$, $H=1$, and $|A|^2-2=0$. It is complete, noncompact, and strongly stable.
The surface threshold thus has a genuine endpoint example. No optimality
claim is made for $\tau_3,\ldots,\tau_6$.
\end{remark}

\begin{proposition}[A stable tube family]\label{prop:tubes}
Let $k\ge2$ and $m\ge1$ be integers, put $n=k+m$, and let $\sigma>0$. In the hyperboloid model
\[
 \mathbb H^{n+1}(-1)\subset\mathbb R^{k,1}\oplus\mathbb R^{m+1},
\]
define
\begin{equation}\label{eq:tubeF}
 F_\sigma(y,z)=\bigl(\sqrt{1+\sigma}\,y,\sqrt\sigma\,z\bigr),
 \qquad (y,z)\in\mathbb H^k(-1)\times\mathbb S^m(1).
\end{equation}
This is a complete, connected, noncompact, boundaryless, two-sided CMC embedding with
\begin{align}
 g_\sigma&=(1+\sigma)g_{\mathbb H^k}+\sigma g_{\mathbb S^m},\label{eq:tubeg}\\
 H_\sigma&=\frac{n\sigma+m}{n\sqrt{\sigma(1+\sigma)}},
 \qquad |A_\sigma|^2-n=\frac m\sigma-\frac k{1+\sigma}.
 \label{eq:tubeH}
\end{align}
It is strongly stable whenever
\begin{equation}\label{eq:tubestability}
 \bigl((k+1)^2-4m\bigr)\sigma\ge4m.
\end{equation}
If $D_{k,m}:=(k+1)^2-4m>0$, equality in \eqref{eq:tubestability} is attained at $\sigma_*=4m/D_{k,m}$, where
\begin{equation}\label{eq:tubeendpoint}
 H_{\sigma_*}^2=
 \frac{m(k^2+6k+1)^2}{4(k+m)^2(k+1)^2}.
\end{equation}
\end{proposition}
\begin{proof}
The Lorentzian inner product of $F_\sigma$ with itself is $-(1+\sigma)+\sigma=-1$. The map is injective and the metric is \eqref{eq:tubeg}, which is complete. The vector field
\[
 N_\sigma=(\sqrt\sigma\,y,\sqrt{1+\sigma}\,z)
\]
has squared length one and is orthogonal to $F_\sigma$ and its tangent space. Differentiating this field along the two product factors gives principal curvatures
\begin{equation}\label{eq:tubecurvatures}
 \alpha=\sqrt{\frac\sigma{1+\sigma}}\quad(k\text{ times}),
 \qquad
 \beta=\sqrt{\frac{1+\sigma}\sigma}\quad(m\text{ times}).
\end{equation}
Indeed, the flat derivative of $N_\sigma$ is tangent to the immersed hypersurface, so it agrees with its hyperbolic covariant derivative. Taking the trace and the squared norm gives \eqref{eq:tubeH}.

To prove strong stability directly, use upper-half-space coordinates $(x,t)\in\mathbb R^{k-1}\times(0,\infty)$ on $\mathbb H^k(-1)$. Its Laplacian is
\[
 \Delta_{\mathbb H^k}=t^2(\Delta_x+\partial_t^2)-(k-2)t\partial_t.
\]
For the positive function $u=t^{(k-1)/2}$, independent of $z$, this yields
\[
 \Delta_{g_\sigma}u=-\lambda_\sigma u,
 \qquad \lambda_\sigma=\frac{(k-1)^2}{4(1+\sigma)}.
\]
Write $q_\sigma=|A_\sigma|^2-n$. If $f$ is compactly supported and smooth, expansion with $f=uv$ and integration by parts give
\begin{align}
 Q_{-1}(f)
 &=\int\bigl(|\nabla f|^2-q_\sigma f^2\bigr)\,\mathrm dV_{g_\sigma}\notag\\
 &=\int u^2|\nabla(f/u)|^2\,\mathrm dV_{g_\sigma}
       +(\lambda_\sigma-q_\sigma)\int f^2\,\mathrm dV_{g_\sigma}.
 \label{eq:tubegroundstate}
\end{align}
The second coefficient is
\begin{equation}\label{eq:tubemargin}
 \lambda_\sigma-q_\sigma
 =\frac{\bigl((k+1)^2-4m\bigr)\sigma-4m}
        {4\sigma(1+\sigma)}.
\end{equation}
Thus \eqref{eq:tubestability} makes both terms in \eqref{eq:tubegroundstate} nonnegative. This proof requires neither an assumed spectral formula nor an integrability condition on $u$, because all test functions have compact support. Substituting $\sigma_*$ in \eqref{eq:tubeH} proves \eqref{eq:tubeendpoint}.
\end{proof}

\begin{corollary}[A noncompact strongly stable example with $H>1$]\label{cor:tube6}
There is a complete noncompact strongly stable CMC embedding into $\mathbb H^7(-1)$ with $H^2=49/48$.
\end{corollary}
\begin{proof}
Take $k=m=3$ and $\sigma=3$. Then
\[
 F(y,z)=(2y,\sqrt3\,z),\qquad
 g=4g_{\mathbb H^3}+3g_{\mathbb S^3},
\]
while
\[
 H=\frac7{4\sqrt3},\qquad |A|^2=\frac{25}{4},
 \qquad q_\sigma=\lambda_\sigma=\frac14.
\]
Equation~\eqref{eq:tubegroundstate} proves strong stability. In fact, all $\sigma\ge3$ with $k=m=3$ are strongly stable, and their mean curvatures satisfy
\[
 H_\sigma^2=1+\frac1{4\sigma(1+\sigma)}>1.
\]
Thus this family approaches the horospherical value from above without becoming compact.
\end{proof}

\begin{corollary}[No dimension-independent hyperbolic threshold]\label{cor:highdimtube}
There is a sequence of complete noncompact strongly stable CMC embeddings
\[
 M^{3m}\longrightarrow\mathbb H^{3m+1}(-1),\qquad m\ge1,
\]
whose squared averaged mean curvatures satisfy
\begin{equation}\label{eq:unboundedTube}
 H_m^2=\frac{(4m^2+12m+1)^2}{36m(2m+1)^2},
 \qquad \lim_{m\to\infty}\frac{H_m^2}{3m}=\frac1{27}.
\end{equation}
Consequently no fixed finite constant $T$ can imply compactness from finite index and $H^2\ge T$ in every dimension.
\end{corollary}
\begin{proof}
In Proposition~\ref{prop:tubes}, take $k=2m$, $n=3m$, and
\[
 \sigma=\sigma_*=\frac{4m}{4m^2+1}.
\]
Then $D_{2m,m}=4m^2+1>0$, and equality holds in \eqref{eq:tubestability}.
The resulting embedding is strongly stable. Formula~\eqref{eq:tubeendpoint} gives the value of $H_m^2$ in \eqref{eq:unboundedTube}. Dividing its numerator and denominator by the corresponding powers of $m$ gives the displayed limit. Hence $H_m^2\to\infty$, while every domain is noncompact and has index zero.
\end{proof}
This conclusion concerns a threshold independent of dimension; it does not exclude a suitable threshold depending on $n$.

\subsection{Dimensional restrictions of the argument}\label{sec:dimensions}
The space-form formulas \eqref{eq:Ric} and \eqref{eq:simons}, the divergence-free tensors, and Proposition~\ref{prop:identities} hold in every dimension $n\ge2$. Likewise, the coefficient relation $\beta=2\chi/(n+2)$ matches the Codazzi--Kato inequality in every dimension, as in \eqref{eq:KatoMatch}. The additional scalar and mixed quartic positivity is proved here only for the coefficient rows with $3\le n\le6$, together with the separate surface argument. Kato matching alone does not establish coercivity in a higher dimension.

The curvature reduction is a separate issue. Lemma~\ref{lem:curvature} uses flatness of complete stable minimal immersions in the relevant Euclidean dimension. Hardt--Simon~\cite{HS85} construct smooth area-minimizing hypersurfaces asymptotic to minimizing cones; their result applied to the Simons cone gives a complete nonflat stable example in $\mathbb R^8$. For $n\ge8$, the nonaffine entire minimal graphs of Bombieri--De Giorgi--Giusti~\cite{BDG69} give further smooth complete examples. Thus this flatness input does not extend to every dimension. Increasing a fixed $|H|$ does not by itself repair the blow-up argument, since $H/\lambda_j\to0$ still holds.

These examples rule out an all-dimensional extension of the strong-stability flatness corollary, but not the weaker assertion that a finite-index noncompact CMC immersion is minimal: the examples themselves have $H=0$. Corollary~\ref{cor:highdimtube} separately rules out a dimension-independent hyperbolic threshold. Higher-dimensional statements may therefore require different coefficients, additional geometric information, or a replacement for the flatness input. The next classical result illustrates an extension with an explicit volume-growth assumption.

\subsection{A classical volume-growth criterion}\label{subsec:entropy}
For a complete noncompact manifold, define
\begin{equation}\label{eq:entropy}
 \mathfrak h(M)=\limsup_{R\to\infty}\frac{\log\Vol_gB_R(o)}R.
\end{equation}
The value is independent of the base point. The entropy criteria of Ilias--Nelli--Soret~\cite[Corollaries~7.1 and~7.2 in the cited preprint]{INS16}, based on the spectrum--growth comparison of Brooks~\cite{Brooks81}, contain the following consequence. Related polynomial-growth hypotheses are treated in \cite{AdC93}.

\begin{proposition}[The known zero-entropy consequence]\label{thm:alln}
Let $n\ge2$, $c\in\{-1,0,1\}$, and let
$F:(M^n,g)\to X_c^{n+1}$ be a smooth, connected, complete, noncompact, boundaryless, two-sided CMC immersion with $H=\tr_gA/n$.
If its strong or volume-constrained Morse index is finite and $\mathfrak h(M)=0$, then
\begin{equation}\label{eq:allnconclusion}
 H^2+c\le0.
\end{equation}
\end{proposition}
A direct proof is included in Appendix~\ref{app:entropy}. More generally, the cited entropy comparison gives
\[
 n(H^2+c)\le\frac{\mathfrak h(M)^2}{4}
\]
for complete noncompact finite-index CMC immersions into a space form. Unlike Theorems~\ref{thm:sphere}--\ref{thm:hyperbolic}, the zero-entropy consequence assumes subexponential intrinsic volume growth. Neither that assumption nor the quantitative entropy comparison is used in the main proof.

\begin{remark}[Complete constant-curvature curves]\label{rem:curves}
The main results concern $n\ge2$, but the one-dimensional conclusion is elementary. A connected complete noncompact one-manifold is the arclength line $\mathbb R$. For a curve of constant signed geodesic curvature $H$ in $X_c^2$,
\[
 Q_c(f)=\int_{\mathbb R}\bigl(|f'|^2-(H^2+c)f^2\bigr)\dd s.
\]
If $H^2+c>0$, a long plateau has negative form, and disjoint translates give infinite strong index. Lemma~\ref{lem:exterior} then also gives infinite constrained index. Thus noncompact finite-index curves have $H^2+c\le0$. In particular the spherical compactness and Euclidean minimality statements hold for curves, and the hyperbolic sufficient condition is $|H|>1$.
\end{remark}

\appendix
\section{Explicit coefficients and positivity margins}\label{app:coefficients}
This appendix gives the finite rational calculations used in Proposition~\ref{prop:coercive}. The geometric and analytic arguments require only the identities and inequalities displayed in the paper. No additional computational input is needed.

\begin{lemma}[Positivity of the selected coefficient rows]\label{lem:arithmetic}
For the four rows of Table~\ref{tab:coefficients}, the coefficients in \eqref{eq:coefficients} satisfy
\[
 \chi,\upsilon,\ell>0,\qquad
 0<\upsilon<\min\{2,n/(n-1)\},\qquad k_n,j_n>0.
\]
The scalar coefficients in \eqref{eq:polynomialCoefficients} have
$c_0,c_4,c_{22},c_{04}>0$, and both discriminants in \eqref{eq:discriminants} are positive. The margins \eqref{eq:margins} satisfy \eqref{eq:marginBounds}.
\end{lemma}
\begin{proof}
Substitution in \eqref{eq:coefficients} gives the following derived coefficients; all fractions are exact.
\begin{center}
\small
\renewcommand{\arraystretch}{1.65}
\begin{tabular}{ccccc}
\toprule
$n$&$\upsilon$&$\mathfrak D$&$\zeta$&$\theta$\\\midrule
3&$\dfrac{28}{75}$&$\dfrac{169}{75}$&$- \dfrac{2873}{4275}$&$\dfrac{21}{169}$\\[5pt]
4&$\dfrac{343}{1000}$&$\dfrac{2971}{1000}$&$- \dfrac{38623}{75000}$&$\dfrac{343}{2971}$\\[5pt]
5&$\dfrac{83}{240}$&$\dfrac{217}{60}$&$- \dfrac{217}{500}$&$\dfrac{415}{3472}$\\[5pt]
6&$\dfrac{173}{480}$&$\dfrac{403}{96}$&$- \dfrac{403}{1056}$&$\dfrac{519}{4030}$\\[5pt]
\bottomrule
\end{tabular}
\end{center}
The other coefficients are given by $\beta=2\chi/(n+2)$, $\epsilon=-(n-2)\alpha$, $k_n=n-1+n\upsilon$, and $j_n=n(2\chi-\beta)$. Their signs and the Hessian admissibility inequalities follow at once from these rows.

The six polynomial coefficients are listed together below. To keep every entry rational, write
$t_{31}=\sqrt{n(n-1)}\,c_{31}$, so $c_{31}^2=t_{31}^2/[n(n-1)]$.
\begin{center}
\small
\renewcommand{\arraystretch}{1.65}
\setlength{\tabcolsep}{5pt}
\begin{tabular}{ccccccc}
\toprule
$n$&$c_0$&$c_2$&$c_4$&$t_{31}$&$c_{22}$&$c_{04}$\\\midrule
3&$\dfrac{74674}{243675}$&$\dfrac{1021}{18825}$&$\dfrac{11422729}{231361000}$&$\dfrac{26811117}{57840250}$&$\dfrac{15849869147}{87107416500}$&$\dfrac{4}{1255}$\\[5pt]
4&$\dfrac{56449}{312500}$&$\dfrac{100841}{207000}$&$\dfrac{82787123}{982658250}$&$\dfrac{481182328}{163776375}$&$\dfrac{1078246301}{502247550}$&$\dfrac{1778}{5175}$\\[5pt]
5&$\dfrac{66877}{600000}$&$\dfrac{1531}{870}$&$\dfrac{127537}{1251222}$&$\dfrac{7373425}{834148}$&$\dfrac{290129923}{30237865}$&$\dfrac{204}{29}$\\[5pt]
6&$\dfrac{161}{2420}$&$\dfrac{213}{50}$&$\dfrac{14931791}{141017760}$&$\dfrac{112942573}{5875740}$&$\dfrac{1710961237}{58757400}$&$\dfrac{342}{5}$\\[5pt]
\bottomrule
\end{tabular}
\end{center}
All denominators are positive. Direct multiplication gives the two discriminants:
\begin{center}
\small
\renewcommand{\arraystretch}{1.65}
\begin{tabular}{ccc}
\toprule
$n$&$\Delta_n^{(h)}$&$\Delta_n^{(bh)}$\\\midrule
3&$\dfrac{370490317}{383794216875}$&$\dfrac{335474248}{2722106765625}$\\[5pt]
4&$\dfrac{58539623731}{5356125000000}$&$\dfrac{52474249003}{12713141109375}$\\[5pt]
5&$\dfrac{934531}{23653125}$&$\dfrac{4336967}{829381440}$\\[5pt]
6&$\dfrac{16551}{302500}$&$\dfrac{40466969}{2350296000}$\\[5pt]
\bottomrule
\end{tabular}
\end{center}
Every numerator in this table is strictly positive. Substitution into \eqref{eq:margins} gives the exact margins below; hence the lower bounds in \eqref{eq:marginBounds} follow by cross-multiplication of positive integers.
\begin{center}
\small
\renewcommand{\arraystretch}{1.85}
\setlength{\tabcolsep}{5pt}
\begin{tabular}{cccc}
\toprule
$n$&$r_n$&$s_n$&$\xi_n$\\\midrule
3&$\dfrac{27665993931658}{640479639917025}$&$\dfrac{335474248}{1981233643375}$&$\dfrac{370490317}{940907334800}$\\[5pt]
4&$\dfrac{3304503219991219}{812739513744062500}$&$\dfrac{4770386273}{9924767088750}$&$\dfrac{58539623731}{7740106243200}$\\[5pt]
5&$\dfrac{62498629687}{88459256100000}$&$\dfrac{30358769}{222819780864}$&$\dfrac{7476248}{168730671}$\\[5pt]
6&$\dfrac{296079}{2956661620}$&$\dfrac{40466969}{273753797920}$&$\dfrac{16551}{161000}$\\[5pt]
\bottomrule
\end{tabular}
\end{center}
This proves all assertions.
\end{proof}

\subsection*{The upper-dimensional row in detail}
For $n=6$, the selected hyperbolic endpoint is $H^2=5/3$, so $\delta_6=3/5$.
The free coefficients $(\alpha,\chi,\ell)=(55/81,31/40,57/2)$ give
\[
 \beta=\frac{31}{160},\quad \epsilon=-\frac{220}{81},\quad
 \upsilon=\frac{173}{480},\quad \mathfrak D=\frac{403}{96},\quad
 \zeta=-\frac{403}{1056},\quad \theta=\frac{519}{4030}.
\]
In particular
\[
 k_6=\frac{573}{80},\qquad j_6=\frac{651}{80}.
\]
For example,
\[
 c_0=\frac15-\frac{173}{1920}
       -\frac{(403/1056)^2\,5}{4(403/96)}
     =\frac{161}{2420},\qquad
 c_2=\frac{57}{8}-\frac25\frac{573}{80}=\frac{213}{50}.
\]
The discriminants are
\[
 \Delta_6^{(h)}=4\frac{161}{2420}\frac{342}{5}
                        -\left(\frac{213}{50}\right)^2
                 =\frac{16551}{302500},
 \qquad
 \Delta_6^{(bh)}=\frac{40466969}{2350296000}.
\]
Consequently
\[
 \xi_6=\frac{16551}{161000},\qquad
 r_6=\frac{296079}{2956661620},\qquad
 s_6=\frac{40466969}{273753797920}.
\]
The smallest scalar margin exceeds the required bound because
\[
 r_6-\frac1{10000}
 =\frac{1032095}{7391654050000}>0.
\]
The two identities \eqref{eq:squares} then give the complete positive decomposition; there is no remaining numerical sign decision.

\section{A direct proof of the zero-entropy criterion}\label{app:entropy}
For completeness, we prove the known consequence stated in Proposition~\ref{thm:alln}. This argument is independent of the low-dimensional coercive combinations.

\begin{proof}[Proof of Proposition~\ref{thm:alln}]
Finite index gives stability outside a compact set $K$ by the disjoint-support argument of Lemma~\ref{lem:exterior}, which is independent of dimension. Suppose $H^2+c>0$ and set $\lambda=n(H^2+c)$. Since $|A|^2\ge nH^2$,
\begin{equation}\label{eq:allnexterior}
 \lambda\int f^2\,\mathrm dV_g\le\int|\nabla f|^2\,\mathrm dV_g,
 \qquad f\in\operatorname{Lip}_c(M\setminus K).
\end{equation}
The dimension-independent volume argument of Lemma~\ref{lem:volume} gives
\[
 \Vol_g B_r(x)\ge\omega_n r^n
       e^{-n\sqrt{H^2+\max\{c,0\}}\,r}.
\]
Thus $M$ has infinite volume, a fact not implied by zero entropy alone.

Choose $a_0>0$ sufficiently large and a smooth function $\chi$ with compactly supported gradient such that $\chi=0$ on a neighborhood of $K$, $\chi=1$ outside $B_{a_0}(o)$, and $0\le\chi\le1$. Put $C_\chi=\int|\nabla\chi|^2\,\mathrm dV_g$. Fix $L>0$. For $R>a_0$, let $\eta_{R,L}$ equal one on $B_R(o)$, decrease linearly to zero on $B_{R+L}(o)\setminus B_R(o)$, and vanish outside $B_{R+L}(o)$. The supports of $\nabla\chi$ and $\nabla\eta_{R,L}$ are disjoint. Apply \eqref{eq:allnexterior} to $f=\chi\eta_{R,L}$ to obtain, with $V(R)=\Vol_g B_R(o)$,
\[
 \lambda\bigl(V(R)-V(a_0)\bigr)
 \le C_\chi+L^{-2}\bigl(V(R+L)-V(R)\bigr).
\]
Since $V(R)\to\infty$, for all sufficiently large $R$ this implies
\[
 V(R+L)\ge\left(1+\frac{\lambda L^2}{2}\right)V(R).
\]
Iteration along $R+jL$ yields
\[
 \mathfrak h(M)\ge L^{-1}\log\left(1+\frac{\lambda L^2}{2}\right)>0,
\]
a contradiction. This proves \eqref{eq:allnconclusion}. The comparison between the strong and constrained indices is unchanged in dimension $n$, proving the last assertion.
\end{proof}

\section*{Acknowledgment of AI assistance}
The author used ChatGPT in this work. This assistance included proposing and
refining candidate proofs, performing computational checks, and
reviewing arguments through separate AI agents. The author guided the research direction, checked the mathematical statements and proofs, and carried out the final revisions. The author takes full responsibility for the content of the paper.

\end{document}